\documentclass[11pt,english]{article}
\usepackage{charter}
\usepackage{helvet}
\usepackage[T1]{fontenc}
\usepackage[latin9]{luainputenc}
\usepackage{geometry}
\usepackage{textcomp}
\usepackage{mathtools}
\usepackage{amsmath}
\usepackage{amsthm}
\usepackage{amssymb}
\usepackage{graphicx}
\usepackage{setspace}
\usepackage{url}
\makeatletter

\numberwithin{equation}{section}
\numberwithin{figure}{section}
\theoremstyle{plain}
\newtheorem{thm}{\protect\theoremname}
  \theoremstyle{plain}
  \newtheorem{prop}[thm]{\protect\propositionname}
  \theoremstyle{plain}
  \newtheorem{cor}[thm]{\protect\corollaryname}
  \theoremstyle{remark}
  \newtheorem*{rem*}{\protect\remarkname}
  \theoremstyle{plain}
  \newtheorem{lem}[thm]{\protect\lemmaname}
  \theoremstyle{definition}
  \newtheorem{example}[thm]{\protect\examplename}
  \theoremstyle{definition}
  \newtheorem{defn}[thm]{\protect\definitionname}
  \theoremstyle{plain}
  \newtheorem*{conjecture*}{\protect\conjecturename}

\usepackage{xcolor}    
\usepackage{colortbl}  
\usepackage{empheq}    
\usepackage{easybmat}  
\usepackage{chngcntr}  
\usepackage{arydshln}  
\usepackage{mathbbol}  
\usepackage{tensor}    
\usepackage{mathtools} 

\DeclareMathOperator{\Tr}{Tr}
\DeclareMathOperator{\adj}{adj}
\newcommand{\incent}[1][k]{\ensuremath{\tensor*[^{#1}]{\vec{\mathbf{v}}}{}}}
\newcommand{\nbcent}[1][k]{\ensuremath{\tensor*{\vec{\mathbf{v}}}{^{#1}}}}

\newcommand{\errata}[1]{{\color{red}[#1]}}

\makeatletter
\@addtoreset{thm}{section}
\makeatother

\makeatother

\usepackage{babel}
\providecommand{\conjecturename}{Conjecture}
\providecommand{\corollaryname}{Corollary}
\providecommand{\definitionname}{Definition}
\providecommand{\examplename}{Example}
\providecommand{\lemmaname}{Lemma}
\providecommand{\propositionname}{Proposition}
\providecommand{\remarkname}{Remark}
\providecommand{\theoremname}{Theorem}

\begin{document}

\title{Non-backtracking Spectrum: Unitary Eigenvalues and Diagonalizability}

\author{Leo Torres}

\date{%
    {\small Network Science Institute,
    Northeastern University\\
    \url{leo@leotrs.com}}
}

\maketitle

\begin{abstract}
Much effort has been spent on characterizing the spectrum of the non-backtracking
matrix of certain classes of graphs, with special emphasis on the
leading eigenvalue or the second eigenvector. Much less attention
has been paid to the eigenvalues of small magnitude; here, we fully
characterize the eigenvalues with magnitude equal to one. We relate
the multiplicities of such eigenvalues to the existence of specific
subgraphs. We formulate a conjecture on necessary and sufficient conditions
for the diagonalizability of the non-backtracking matrix. As an application,
we establish an interlacing-type result for the Perron eigenvalue.
\end{abstract}
\tableofcontents{}

\section{Introduction}

\counterwithin{thm}{section}

A walk is called \emph{backtracking} if it returns to a node immediately
after leaving it, i.e. if it contains a sub-walk of the type $i\!\to\!j\!\to\!i$.
The \emph{non-backtracking matrix} is the transition matrix of a random
walker that does not perform backtracks, and it has received much
attention lately. The main hurdle in studying the eigenvalue spectrum
of the non-backtracking matrix is that it is not normal. This means
that many standard tools in spectral graph theory do not apply to
it as some of them apply only to symmetric matrices such as the adjacency
and Laplacian matrices. In view of the spectral theorem, non-normality
implies that the non-backtracking matrix does not admit a unitary
basis of eigenvectors. However, it may still admit a basis of eigenvectors
that is non-unitary or, equivalently, it may be diagonalized by a
non-unitary matrix. In this work we study this possibility. For simplicity,
we use the ``NB-'' prefix to mean ``non-backtracking''. For example,
we use \emph{NB-matrix }and \emph{NB-eigenvalue} to refer to the matrix
and\emph{ }to one of its eigenvalues, respectively. All graphs considered
are simple, undirected, unweighted, and connected.

We study the diagonalizability of the NB-matrix by considering three
different types of graphs: those containing zero cycles (i.e. trees),
exactly one cycle, and two or more cycles. These graphs allow for
different long-term behaviors of NB-walks, which are codified in the
NB-eigenvectors. Indeed, if the graph is a tree, every NB-walk will
die out as soon as it reaches a node of degree one. Accordingly, every
NB-eigenvalue of a tree is zero and the NB-matrix is never diagonalizable.
If the graph contains exactly one cycle then every NB-walk must either
die out eventually or continue to go around the cycle forever. Accordingly,
the NB-spectrum of these graphs contains a cyclic group, namely the
$n^{th}$ roots of unity where $n$ is the number of nodes in the
cycle. Further, the NB-matrix of a cycle graph (a.k.a. circle graph)
is a block-permutation matrix, which is always diagonalizable. Lastly,
if the graph contains two or more cycles then the NB-walks may have
complex long-term behaviors and, accordingly, the NB-eigenvalues no
longer have a straightforward characterization as in the previous
two cases. In this latter case, we find that under mild assumptions,
and assuming a conjecture we formulate later, the NB-matrix is diagonalizable,
and we exhibit some of the properties of the basis of eigenvectors.

Our approach to study graphs with at least two cycles is based on
the fact that a matrix is diagonalizable if and only if each of its
eigenvalues has equal algebraic and geometric multiplicities. We study
the multiplicities of each possible eigenvalue according to its magnitude.
Let the graph $G$ be given and let $\lambda$ be a NB-eigenvalue
of $G$. If $|\lambda|<1$, we say $\lambda$ is an ``inner'' eigenvalue,
while if $1<|\lambda|<\rho$ we say $\lambda$ is ``outer''; here
$\rho$ is the spectral radius of the matrix. If $|\lambda|=1$ we
call it ``unit'' or ``unitary'', and finally if $|\lambda|=\rho$,
we say $\lambda$ is a ``leading'' eigenvalue; see Figure \ref{fig:eigenvalues-by-magnitude}.
The multiplicities of inner and leading eigenvalues are well known,
though here we revisit these results for completeness. The case of
the eigenvalues $\lambda=\pm1$ is also well-known.

\begin{figure}
\begin{centering}
\includegraphics[scale=0.9]{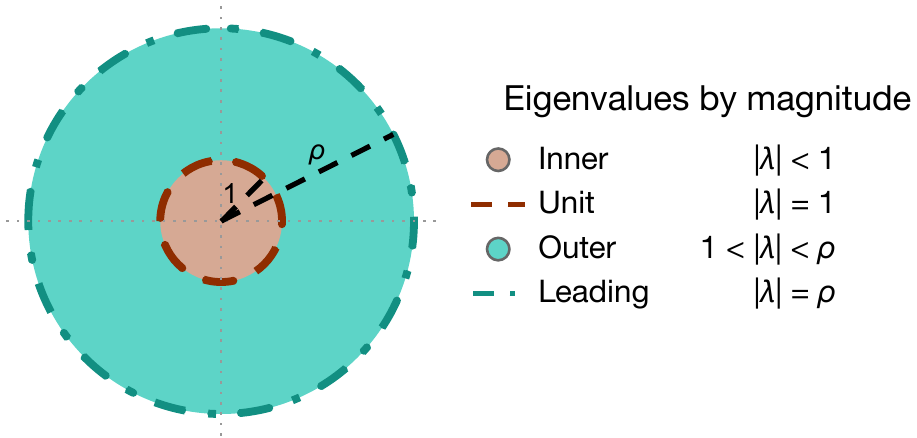}
\par\end{centering}
\caption{\label{fig:eigenvalues-by-magnitude}Eigenvalue categorization by
magnitude in the complex plane.}
\end{figure}

We thus focus on the unit and outer eigenvalues. The main contribution
of this work is two-fold: we compute the multiplicities of complex
unitary eigenvalues, and conjecture that in most cases the multiplicity
of outer eigenvalues is always one, and thus they do not pose a hurdle
to diagonalizability. 

In the case of unit eigenvalues, we explicitly compute the unit eigenvalues
and corresponding eigenvectors for any arbitrary graph. First, we
prove that if a NB-eigenvalue $\lambda$ is unitary, then $\lambda$
must be a root of unity. In other words, there are no unit NB-eigenvalues
with irrational argument. Then, we show that eigenvectors of unit
eigenvalues are localized to specific subgraphs (a.k.a. \emph{motifs}).
Consider a set of nodes $\mathcal{C}$ of $r$ nodes in $G$. $G$
will have a unit eigenvalue associated to $\mathcal{C}$, and the
corresponding eigenvector will be supported on $\mathcal{C}$ (i.e.
it will be zero outside of $\mathcal{C}$), if one of the following
holds; see Figure \ref{fig:collar-pendant}.
\begin{enumerate}
\item \label{enu:odd}If $\mathcal{C}$ induces a cycle, $r$ is odd, and
all nodes in $\mathcal{C}$ have degree $2$ in $G$, except for exactly
one node which may have arbitrary degree. In this case, $\mathcal{C}$
is called a \emph{pendant} of size $r$.
\item \label{enu:even}If $\mathcal{C}$ induces a cycle, $r$ is even,
and all nodes in $\mathcal{C}$ have degree $2$ in $G$, except perhaps
for two diametrically opposite nodes which may have arbitrary degrees.
(These two nodes may or may not be neighbors of each other.) In this
case, $\mathcal{C}$ is called a \emph{collar} of size $r$. 
\item \label{enu:bracelet}If $\mathcal{C}$ induces a ``figure eight''
graph made of two cycles of the same length joined at one node, $r$
is even, and all nodes in $\mathcal{C}$ have degree $2$ in $G$,
except perhaps for the one node at which the two cycles meet, which
may have arbitrary degree. In this case, $\mathcal{C}$ is called
a \emph{bracelet} of size $r$. Note a bracelet can be considered
a degenerate form of a collar.
\end{enumerate}
If $G$ contains a set $\mathcal{C}$ that is collar, a pendant, or
a bracelet of size $r$, then the $r^{th}$ roots of unity will all
be NB-eigenvalues of $G$, and the corresponding eigenvectors will
be supported on $\mathcal{C}$. We prove this result in Section (\ref{sec:complex-roots}).

\begin{figure}
\begin{centering}
\includegraphics{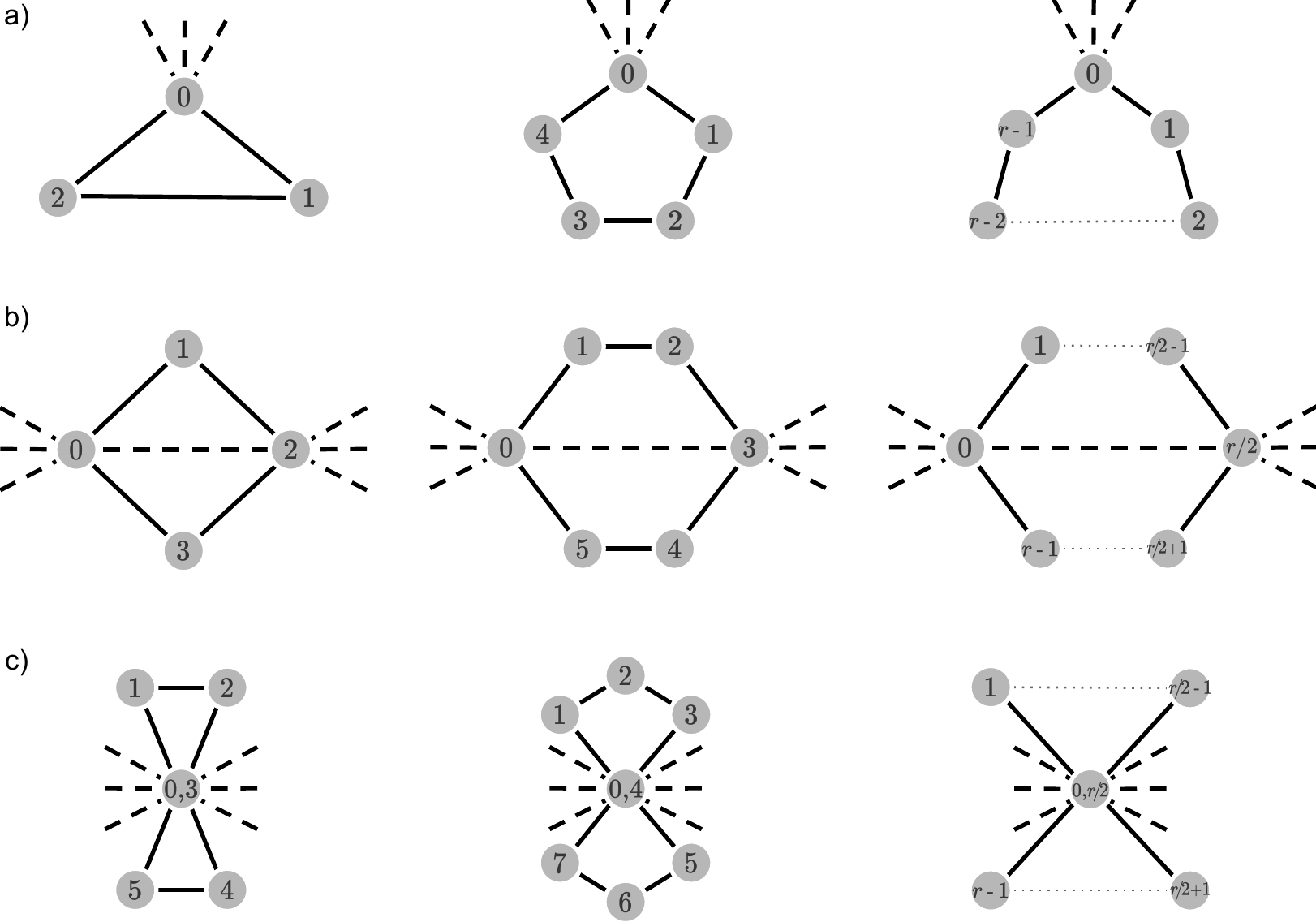}
\par\end{centering}
\caption{\label{fig:collar-pendant}Motifs associated to unit eigenvalues.
Dashed lines denote possible edges. Dotted lines denote missing nodes
all of which have degree $2$. \textbf{a)} Pendants of size $3,5,r$.
Node $0$ is the only one that may have degree larger than $2$. \textbf{b)}
Collars of size $4,6,r$. The nodes $0$ and $r/2$ may have arbitrary
degrees, and they may even be neighbors of each other. \textbf{c)}
Bracelets of size $6,8,r$. A bracelet can be considered as a degenerate
case of a collar where nodes $0,r/2$ have been identified.}
\end{figure}

\begin{figure}
\begin{centering}
\includegraphics{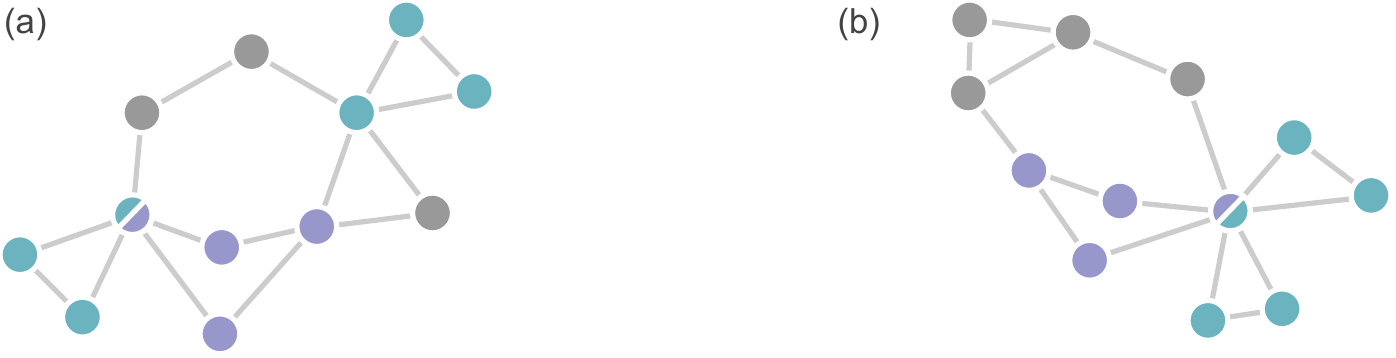}
\par\end{centering}
\caption{\label{fig:crab-squid}
\errata{Errata: Since the original upload of this manuscript, we have found that the graph in (a) has unit eigenvalues that are not explained by our results here. In particular, it has sixth roots of unity without having a bracelet or collar of size $6$. Future versions of this manuscript will deal with this edge case.} Two cospectral graphs, the "crab" \textbf{(a)} and the "squid" \textbf{(b)}; see \cite{durfee2015distinguishing}. They each have two pendants of size $3$ (green nodes) as well as one collar of size $4$ (purple nodes).
}
\end{figure}

In the case of outer eigenvalues, we formulate a conjecture about
the conditions under which they all have multiplicity one. Table \ref{tab:results}
shows the algebraic multiplicity $AM(\lambda)$ and geometric multiplicity
$GM(\lambda)$ of each eigenvalue $\lambda$ in the case of graphs
with at least two cycles.\footnote{Contrast to Table $6.1$ in \cite{Grindrod2018} which deals with
the multiplicities of eigenvalues of a closely related matrix, the
so-called \emph{deformed graph Laplacian}.} All together, our results show that the only eigenvalue $\lambda$
for which $AM(\lambda)$ may not coincide with $GM(\lambda)$ is $\lambda=0$.
Under mild assumptions relating to it, the NB-matrix is diagonalizable. 

Finally, by way of application, we establish a form of eigenvalue
interlacing for the unique real NB-eigenvalue of maximum modulus,
a.k.a the Perron eigenvalue of the NB-matrix. This is done by using
the diagonalizability of the NB-matrix to diagonalize its \emph{resolvent}.
Then, we use standard tools over this resolvent, such as the Perron-Frobenius
theorem and Gershgorin's disk theorem, to prove that the Perron eigenvalue
can only increase when a new node is added to the graph.

We start by reviewing some preliminary facts in Section \ref{sec:preliminaries}.
We being our discussion by fully characterizing the NB-spectrum of
trees in Section \ref{sec:trees}. In Section \ref{sec:1-shell} we
discuss how the tree-like parts of arbitrary graphs have no influence
in the non-zero part of the spectrum and therefore from then on we
focus on graphs with minimum degree at least $2$, that is, graphs
with no tree-like parts. In Section \ref{sec:1-cycle} we characterize
the full spectrum of cycle graphs. In Section \ref{sec:2-cycles}
we discuss the inner, unit, outer, and leading eigenvalues of graphs
with two or more cycles. We review known results for inner and leading
eigenvalues in Sections \ref{sec:inner} and \ref{sec:leading}, respectively,
while our main contributions for unit and outer eigenvalues are found
in Sections \ref{sec:unit} and \ref{sec:outer}, respectively. Finally,
in Section \ref{sec:application} we use this knowledge to study the
Perron eigenvalue after adding a new node to the graph.

\begin{table}
\begin{centering}
\begin{tabular}{c|c|c|c}
Category & Sub-category & GM($\lambda)$ (AM($\lambda$); if different) & Section\tabularnewline
\hline 
Inner & $\lambda=0$ & $n_{1}$ ($2s_{1}$) & \ref{sec:1-shell}\tabularnewline
 & $0<|\lambda|<1$ & impossible & \ref{sec:inner}\tabularnewline
\hline 
Unit & $\lambda^{r}=1$, even $r$ & number of ``collars'' or ``bracelets'' & \ref{sec:unit}\tabularnewline
 & $\lambda^{r}=1$, odd $r$ & number of ``pendants'' & \ref{sec:unit}\tabularnewline
 & $\lambda^{r}\neq1,\forall r\in\mathbb{Z}$ & impossible & \ref{sec:unit}\tabularnewline
 & $\lambda=1$ & $m-n+1$ & \ref{sec:real-roots}\tabularnewline
 & $\lambda=-1$ & $m-n$ & \ref{sec:real-roots}\tabularnewline
\hline 
Outer & $1<|\lambda|<\rho$ & $1$ (\emph{conjecture}) & \ref{sec:outer}\tabularnewline
\hline 
Leading & $|\lambda|=\rho$ & $1$ & \ref{sec:leading}\tabularnewline
\end{tabular}
\par\end{centering}
\caption{\label{tab:results}Geometric multiplicity (GM) and algebraic multiplicity
(AM), if different, of NB-eigenvalues on graphs with at least two
cycles. $n_{1}$ is the number of nodes of degree one, $s_{1}$ is
the number of nodes in the $1$-shell.}
\end{table}

\section{Preliminaries and notation\label{sec:preliminaries}}

\paragraph{Generalities}

All graphs considered are undirected, simple, connected, and contain
at least $2$ nodes. For a node $i$ in $G$, we write $d_{i}$ for
its degree, i.e. the number of neighbors in $G$. If the minimum degree
of $G$ is at least $x$ we say $G$ is ``md$x$''. If $S$ is a
set of nodes of $G$, by $G\setminus S$ we mean the subgraph induced
by all nodes except those in $S$. In Appendix \ref{app:eigen} we
recall standard nomenclature relating to eigenvalues and eigenvectors.
We will also make use of the two following concepts: the \emph{$2$-core
of $G$ }is the maximal induced subgraph of $G$ in which each node
has degree at least $2$, whereas the \emph{$1$-shell of $G$ }is
the graph induced by all those nodes outside the $2$-core. The $1$-shell
is always a forest, and we sometimes refer to it as \emph{the tree-like
parts of $G$}. The nodes in the $1$-shell can be further broken
up into \emph{layers}: the nodes of degree $1$ make up the first
layer, while their neighbors make up the second layer. In general,
the neighbors of the nodes in the $r^{th}$ layer that are in the
$1$-shell but not in any other layer $s$ for $s<r$ make up the
$(r+1)^{th}$ layer. We will usually refer to the nodes in the $1$-shell
as $S$, and to the $2$-core of $G$ as $G\setminus S$. In Figure
\ref{fig:2-core} and Appendix \ref{app:kcore} we expand upon these
definitions and other relevant concepts.

\begin{figure}
\includegraphics[width=1\textwidth]{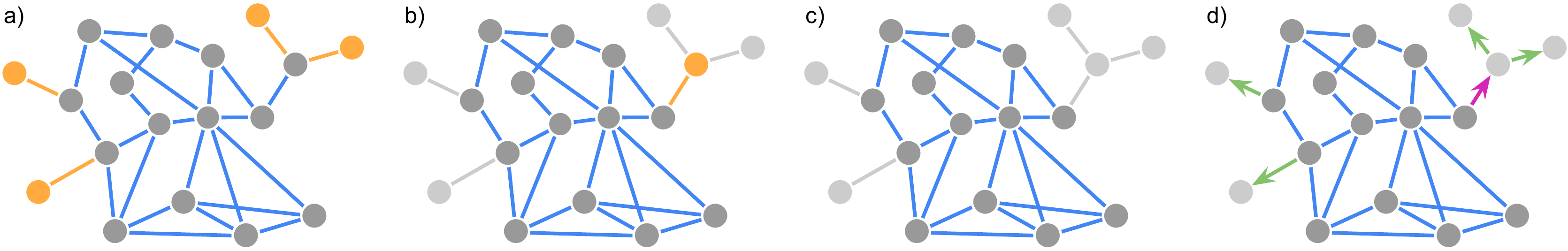}

\caption{\label{fig:2-core}\textbf{a)} The first layer of the $1$-shell of
a graph is highlighted in orange. \textbf{b)} The second layer of
the $1$-shell is highlighted; the first layer is grayed out. \textbf{c)}
The $1$-shell, a forest, is grayed out; what is left is the $2$-core.
\textbf{d)} The characteristic vectors of the green edges lie in the
kernel of $B$. The characteristic vector of the magenta edge lies
in the kernel of $B^{2}$.}
\end{figure}

\paragraph*{Oriented edges and NB-walks}

Let $G$ be a (undirected, unweighted, simple, connected) graph with
$n$ nodes and $m$ edges. Let $E$ be the set of undirected edges
of $G$: if nodes $u$ and $v$ are joined by an edge, we write $u-v$.
Let $\bar{E}$ be the set of \emph{oriented edges} of $G$ and write
$i\to j\in\bar{E}$ for the oriented edge from node $i$ to node $j$.
We say that $i$ is the \emph{source }and $j$ is the \emph{target}
of $i\to j$. Unless specified otherwise, all vectors in this work
are indexed by $\bar{E}$, and we write $\mathbf{v}_{i\to j}$ for
the value of the vector $\mathbf{v}$ at the oriented edge $i\to j$.
We write $\chi^{i\to j}$ for the characteristic vector of $i\to j$,
that is $\chi_{i\to j}^{i\to j}=1$, while $\chi_{e}^{i\to j}=0$
for any oriented edge $e$ different than $i\to j$. 

A \emph{walk} is a sequence of pairwise incident oriented edges, $u_{1}\to v_{1},u_{2}\to v_{2},\ldots,u_{r}\to v_{r}$,
where $v_{s}=u_{s+1}$ for $s=1,\ldots,r-1$. Here, $r$ is the \emph{length
}of the walk. A walk is \emph{closed} if $v_{r}=u_{1}$. A walk is
said to \emph{extend} another walk when the source node of the first
edge of the former walk is the target of the last edge of the latter
walk. The walk $u\to v,v\to u$ is called a \emph{backtrack}, i.e.
if it traces the same edge in different directions one after the other.
A walk of arbitrary length is called a \emph{non-backtracking walk}
if it does not contain backtracks. A closed walk is called a \emph{non-backtracking
cycle} if it is a closed non-backtracking walk and, additionally,
its first and last edges are not a backtrack. Note that both NB-walks
and NB-cycles may be self-intersecting. By abuse of notation, we also
use \emph{cycle} to refer to a set of nodes whose induced subgraph
is a cycle graph (a.k.a. circle graph).

The \emph{NB-matrix} of $G$ is a $2m\times2m$ matrix indexed in
the rows and columns by $\bar{E}$. It is defined as 
\begin{equation}
B_{k\to l,i\to j}\coloneqq\delta_{jk}\left(1-\delta_{il}\right).\label{eqn:nbm}
\end{equation}
 $B$ can be understood as the (unnormalized) transition matrix of
a random walker that does not trace backtracks. That is, $B_{k\to l,i\to j}$
is equal to $1$ whenever $k\to l$ extends $i\to j$ without forming
a backtrack. The action of $B$ on a vector $\mathbf{v}$ represents
the aggregation of all incoming edges, except for the backtrack (see
Figure \ref{fig:nbm-doodle}):

\begin{equation}
\left(B\mathbf{v}\right)_{k\to l}=\sum_{i}a_{ik}\mathbf{v}_{i\to k}-\mathbf{v}_{l\to k}.\label{eqn:nbm-action}
\end{equation}
Similarly, The powers of $B$ count the number of NB-walks: $B_{k\to l,i\to j}^{p}$
is equal to the number of NB-walks that start with $i\to j$ and end
with $k\to l$ with length $p+1$; see Figure \ref{fig:nbm-doodle}(c).

\begin{figure}
\centering{}\includegraphics[width=0.9\textwidth]{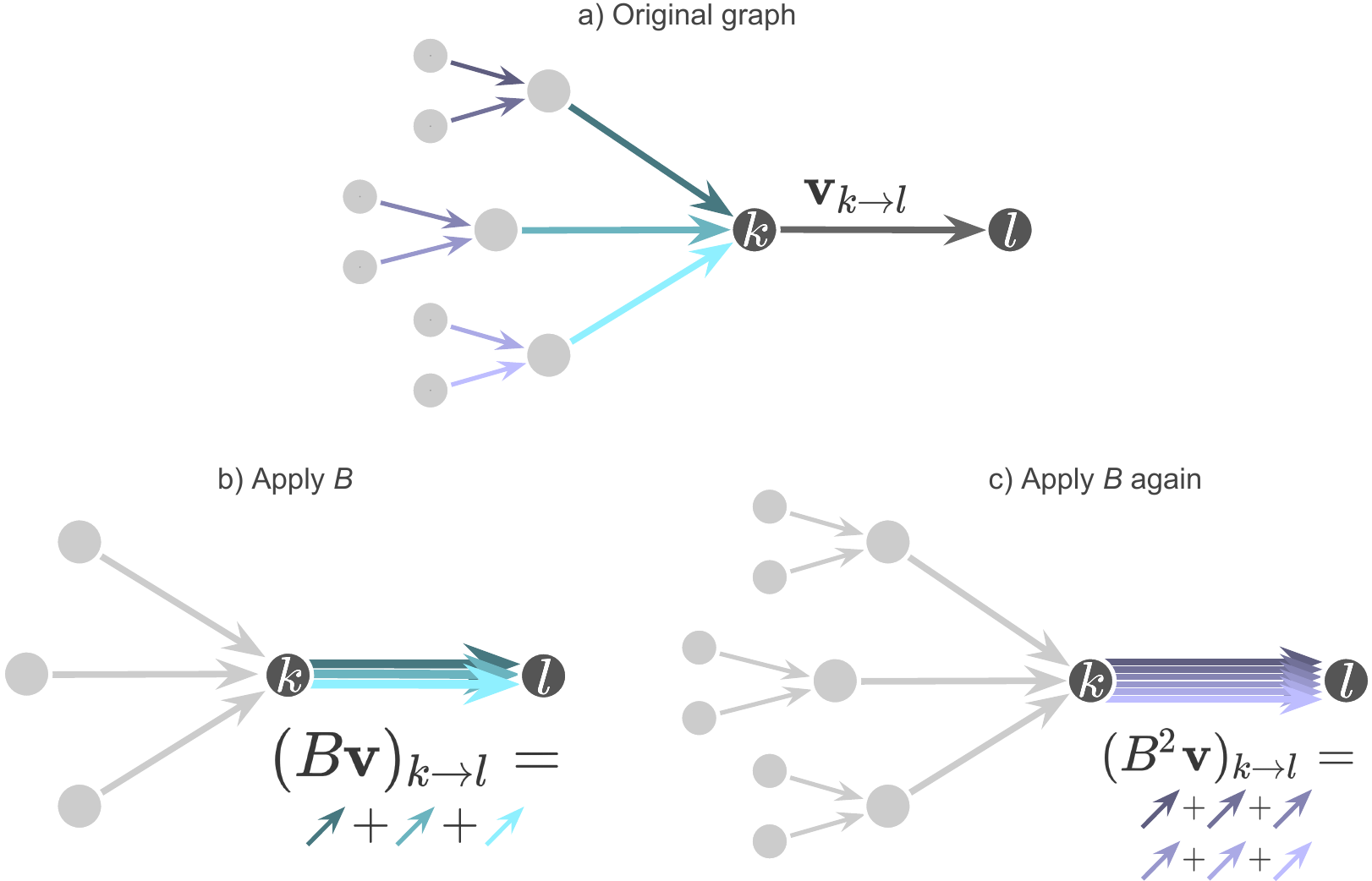}\caption{\label{fig:nbm-doodle}\textbf{Top:} $B\mathbf{v}$ aggregates the
values along all incoming edges, except for the backtrack, i.e., except
for $\mathbf{v}_{l\to k}$. \textbf{Bottom:} $B^{2}\mathbf{v}$ aggregates
the values along all NB-paths of length $3$.}
\end{figure}

\paragraph*{NB-eigenvalues}

$B$ is not symmetric and thus its eigenvalues are in general complex
numbers. Further, $B$ is not normal and thus it cannot be diagonalized
by a unitary matrix. The famous Ihara-Bass determinant formula \cite{kempton2016non,stark1996zeta}
says that if $A$ is the adjacency matrix of $G$ and $D$ is the
diagonal degree matrix, then 
\begin{equation}
\det(I\text{\textminus}tB)=\left(1\text{\textminus}t^{2}\right){}^{m-n}\det\left(I\text{\textminus}tA+t^{2}(D\text{\textminus}I)\right).\label{eqn:ihara-bass}
\end{equation}
Note that the algebraic multiplicity (AM) of a complex number $\lambda$
as an eigenvalue of $B$ equals the multiplicity of $1/\lambda$ as
a root of $\det\left(I-tB\right)$.

Let $\rho$ be the spectral radius of $B$ and recall $\lambda$ is
a leading eigenvalue of $B$ if $|\lambda|=\rho$. Perron-Frobenius
theory determines conditions under which there is one leading eigenvalue
that is positive and real. We call this the \emph{Perron eigenvalue
}of $B$.

Lastly, suppose $B\mathbf{v}=\lambda\mathbf{v}$, and let $k$ and
$l$ be any pair of neighbors in $G$. From (\ref{eqn:nbm-action})
we get

\begin{equation}
\lambda\mathbf{v}_{k\to l}+\mathbf{v}_{l\to k}=\sum_{i}a_{ik}\mathbf{v}_{i\to k}.\label{eqn:eigenvector-in-two-directions}
\end{equation}
When $\lambda$ is the Perron eigenvalue and $\mathbf{v}$ the corresponding
right eigenvector, the right-hand side is called the \emph{NB-centrality}
of $k$ \cite{martin2014localization}, denoted here by $\nbcent$
, 
\begin{equation}
\nbcent\coloneqq\sum_{i}a_{ik}\mathbf{v}_{i\to k}.\label{eqn:nb-centrality}
\end{equation}

\section{Graphs with zero or one cycles\label{sec:zero-one}}

In this Section we provide a complete description of the NB-eigenvalues
and NB-eigenvectors of trees. Then, we show that the $1$-shell of
an arbitrary graph does not influence the non-zero NB-eigenvalues
because the $1$-shell is always a forest, and hence its contribution
to the NB-spectrum can be reduced to the tree case. For this reason,
after this section we will always assume that a graph is md$2$ or,
equivalently, has empty $1$-shell. We also provide a complete description
of the spectrum of graphs with exactly $1$ cycle and empty $1$-shell,
i.e. cycle graphs. The unit NB-eigenvalues of graphs with two or more
cycles are tightly related to the eigenvalues of cycle graphs.

\subsection{Trees\label{sec:trees}}

If $G$ is a tree, as soon as a NB-walk reaches a node of degree one,
it cannot be extended without backtracking. This immediately leads
us to our first result. 
\begin{prop}
\label{pro:tree-eigenvalues}If $G$ is a tree. $B$ is not diagonalizable.
\end{prop}

\begin{proof}
Let $n$ be the number of nodes of $G$. A walk of length $n+1$ must
visit at least one node more than once. However, a NB-walk in a tree
cannot visit any node more than once since there are neither cycles
nor backtracks. Therefore there are no NB-walks of length $n+1$ and
$\mbox{\ensuremath{B^{n}=0}}$. This means that $B$ is nilpotent
or, equivalently, that all of its eigenvalues are zero. Lastly, a
nilpotent matrix is diagonalizable only when it equals the zero matrix,
which is impossible since $G$ is connected.
\end{proof}
Now, the kernels of $B,B^{2},B^{3},\ldots$, track the composition
of the $1$-shell of $G$ in its successive layers. See Figure \ref{fig:2-core}(d)
for an example.
\begin{prop}
\label{pro:tree-eigenvectors}Let $i\to j$ be in the $\ell^{th}$
layer of the $1$-shell of $G$. Then, $B^{\ell}\chi^{i\to j}=0$.
\end{prop}

\begin{proof}
Let $i\to j$ be in the $1^{st}$ layer of the $1$-shell. Equation
(\ref{eqn:nbm}) implies $B\chi^{i\to j}=0$. By induction, suppose
the theorem is true for $\ell\!-\!1$, and let $i\to j$ be in the
$\ell^{th}$ layer. Using (\ref{eqn:nbm}) again we have 
\begin{equation}
B\chi^{i\to j}=\sum_{k\neq i}\chi^{j\to k}.
\end{equation}
However, each $\chi^{j\to k}$ is now in the $\left(\ell-1\right)^{th}$
layer and thus in the kernel of $B^{\ell-1}$.
\end{proof}
Note that since $G$ is a tree, it is equal to its $1$-shell and
thus the last two Propositions complete the characterization of the
eigenvalues and eigenvectors of any tree. However, Proposition \ref{pro:tree-eigenvectors}
applies to any $G$, not just trees. This is the fundamental fact
that we use next.

\subsection{The $1$-shell of arbitrary graphs\label{sec:1-shell}}

Suppose that $G$ has non-empty $2$-core (i.e. it is not a tree)
and non-empty $1$-shell (i.e. it has at least one node of degree
one). Let $i$ have degree $1$ and let $j$ be its neighbor. Then,
$B$ can be written as 

\begin{align}\label{eqn:nbm-block}
B =
\left(
\begin{BMAT}(e)[0pt,0cm,0cm]{ccc.c.c}{ccc.c.c}
 &              & &            &   \\
 & B'           & & \mathbf{0} & D \\
 &              & &            &   \\
 & E^T          & & 0          & 0 \\
 & \mathbf{0}^T & & 0          & 0 \\
\end{BMAT}
\right),
\end{align}where $D$, $E$, and $\boldsymbol{0}$ are column vectors, and $B'$
is the NB-matrix of $G\setminus\{i\}$. Using the theory of Schur
complements (see e.g \cite{horn2012matrix} Equation 0.8.5.1), we
have

\begin{align}\label{eqn:schur-one-node}
\det \left( B - tI \right) = 
\left|
\begin{BMAT}(r)[3pt,0cm,0cm]{ccc.c.c}{ccc.c.c}
 &             & &            &    \\
 & B' - tI     & & \mathbf{0} & D  \\
 &             & &            &    \\
 & E^T         & & -t         & 0  \\
 & \mathbf{0}^T & & 0         & -t \\
\end{BMAT}
\right| =
t^2 \det \left( B' - tI + \frac{1}{t} \left(
\begin{BMAT}(e){c.c}{ccc} & \\ \mathbf{0} & D \\ & \end{BMAT} \right)
\left(
\begin{BMAT}(e){ccc}{c.c} & E^{T} & \\ & \mathbf{0}^T & \end{BMAT} \right)
\right) =
t^2 \det \left( B' - tI \right).
\end{align}In other words, the spectrum of $B$ is exactly that of $B'$ plus
two additional zeros. Now assume $B\mathbf{v}=\lambda\mathbf{v}$
for non-zero $\lambda$ and write $ \mathbf{v} = \left( \begin{BMAT}(@){ccc.c.c}{c} \quad & \mathbf{v}' & \quad & \mathbf{v}_{j \to i} & \mathbf{v}_{i \to j} \end{BMAT} \right)^T $
, so it has the same block-structure as in (\ref{eqn:nbm-block}).
In this case we have 

\begin{align}
\left(
\begin{BMAT}(e)[0pt,0cm,0cm]{ccc.c.c}{ccc.c.c}
 &              & &            &   \\
 & B'           & & \mathbf{0} & D \\
 &              & &            &   \\
 & E^T          & & 0          & 0 \\
 & \mathbf{0}^T & & 0          & 0 \\
\end{BMAT}
\right)
\left(
\begin{BMAT}(r){c}{ccc.c.c} \quad \\ \mathbf{v}' \\ \quad \\ \mathbf{v}_{j \to i} \\ \mathbf{v}_{i \to j} \end{BMAT}
\right) =
\left(
\begin{BMAT}(r){c}{ccc.c.c} \quad \\ B' \mathbf{v}' + \mathbf{v}_{i\to j} D \\ \quad \\ E^{T} \mathbf{v}' \\ 0 \end{BMAT}
\right) = 
\lambda \left(
\begin{BMAT}(r){c}{ccc.c.c} \quad \\ \mathbf{v}' \\ \quad \\ \mathbf{v}_{j \to i} \\ \mathbf{v}_{i \to j} \end{BMAT}
\right),
\end{align}which immediately reduces to

\begin{equation}
\begin{cases}
B'\mathbf{v}' & =\lambda\mathbf{v}'\\
\mathbf{v}_{i\to j} & =0\\
\lambda\mathbf{v}_{j\to i} & =E^{T}\mathbf{v}'
\end{cases}\label{eqn:eigenvector-extension}
\end{equation}
Therefore, if we can find an eigenvector $\mathbf{v}'$ of $B'$,
we can use it to find an eigenvector $\mathbf{v}$ of $B$. Iterating
the above arguments over each node of the $1$-shell yields the following
result.
\begin{prop}
\label{pro:reduction-to-2-core}Let $S$ be the set of nodes in the
$1$-shell of $G$. The non-zero NB-eigenvalues are determined solely
by $G\setminus S$, and all eigenvectors can be computed starting
from an eigenvector of $G\setminus S$.
\end{prop}

\begin{proof}
Let $i$ be a node of degree one of $G$. The arguments in this section
show that the non-zero eigenvalues depend only on $G\setminus\{i\}$,
and that the eigenvectors can be computed using (\ref{eqn:eigenvector-extension}).
Now let $S_{1},S_{2}$ be the set nodes in the $1^{st}$ and $2^{nd}$
layers of the $1$-shell, respectively. Apply the above argument to
each node in $S_{1}$ in turn to show that the eigenvalues are solely
determined by $G\setminus S_{1}$, and the same can be said for the
eigenvectors. But now the nodes in $S_{2}$ have degree one in $G\setminus S_{1}$.
An inductive argument finishes the proof.
\end{proof}
We will have much more to say about the non-zero eigenvalues and their
eigenvectors. However, Proposition \ref{pro:reduction-to-2-core}
establishes that, in order to do so, it is enough to focus on the
$2$-core of a graph. For now, we fully characterize the zero eigenvalue
and kernel of arbitrary graphs.
\begin{prop}
\label{pro:zero}Let $G$ be an arbitrary graph and let $S$ be the
set of nodes in its $1$-shell, with $s_{1}=|S|$, and let $n_{1}$
be the number of nodes of degree one. We have $AM(0)=2s_{1}$, and
$GM(0)=n_{1}$. 
\end{prop}

\begin{proof}
When $\lambda=0$ and $Bv=0$, Equations (\ref{eqn:eigenvector-in-two-directions})
and (\ref{eqn:nb-centrality}) together show that 
\begin{equation}
\mathbf{v}_{l\to k}\left(d_{k}-1\right)=0,\label{eqn:kernel}
\end{equation}

for any oriented edge $l\to k$, where $d_{k}$ is the degree of $k$.
Thus, $\mathbf{v}$ can only be non-zero when there exists at least
one node of degree one in the graph. This shows $GM(0)\geq n_{1}$.

Now, iterating (\ref{eqn:schur-one-node}) over each element of $S$,
in ascending order of layers, shows that $AM(0)$ is exactly equal
to $2s_{1}$ plus the algebraic multiplicity of $0$ in $G\setminus S$.
However, (\ref{eqn:kernel}) shows that $G\setminus S$ never has
$0$ as an eigenvalue since it does not have nodes of degree one.
Therefore, $AM(0)=2s_{1}$. Further, Equation (\ref{eqn:kernel})
shows that there is exactly one vector in the kernel for each node
of degree one and therefore $GM(0)=n_{1}$.
\end{proof}
\begin{cor}
$B$ is invertible if and only if the $1$-shell of $G$ is empty.
In that case, it is given by 
\[
B_{k\to l,i\to j}^{-1}=\frac{\delta_{il}}{d_{l}-1}\left(1-\delta_{kj}\left(d_{l}-1\right)\right).
\]
\end{cor}

\begin{proof}
The first statement is direct from the preceding Proposition. The
second statement can be checked manually using Equation (\ref{eqn:nbm}).
\end{proof}
This finalizes the characterization of $\lambda=0$ in the general
case. For the purpose of diagonalizability, note that $0$ is defective
unless the $1$-shell is empty. 

\subsection{Graphs with one cycle\label{sec:1-cycle}}

Starting now and in the rest of the paper, we assume $G$ is md$2$.
We now focus on graphs with one cycle whose $1$-shell is empty, i.e.
cycle graphs. Let $G$ be a cycle graph with $n$ nodes. In this case,
$B$ has the block form

\begin{equation}
B=\left(\begin{array}{cc}
B^{cw} & \boldsymbol{0}\\
\boldsymbol{0} & B^{ccw}
\end{array}\right),
\end{equation}
where $B^{cw}\,\left(B^{ccw}\right)$ is indexed by the oriented edges
going around the cycle in clockwise (resp. counter-clockwise) order,
and are therefore matrices representing cyclic permutations of order
$n$. 
\begin{prop}
Let $G$ be a cycle graph with $n$ nodes. Then, the eigenvalues of
$B$ are the $n^{\text{th}}$ roots of unity, each with (algebraic
and geometric) multiplicity $2$. $B$ is diagonalizable.
\end{prop}

\begin{proof}
Results on eigenvalues of permutation matrices can be found in standard
references.
\end{proof}
Cycle graphs are important not only because they can be fully characterized,
but because NB-eigenvalues that are roots of unity are essential to
our later discussion. They always appear, in any graph, and are related
to the existence of collars, pendants, and bracelets (see Figure \ref{fig:collar-pendant}).
Note that every cycle graph is itself a collar or a pendant.

\subsection{Examples\label{sec:examples-tree-circle}}

Figure \ref{fig:tree-circle}(a) shows a tree with two layers. Per
Proposition \ref{pro:tree-eigenvalues}, all its NB-eigenvalues are
zero. Per proposition \ref{pro:tree-eigenvectors}, the characteristic
vectors of the orange edges lie in the kernel of $B$, while the characteristic
vectors of the green and blue edges lie in the kernels of $B^{2}$
and $B^{3}$, respectively. Per Proposition \ref{pro:zero}, we have
$GM(0)=3$ and $AM(0)=8$. Figure \ref{fig:tree-circle}(b) shows
a graph with one cycle and non-empty $1$-shell. Note the $1$-shell
is isomorphic to the graph in (a), and the cycle is a pendant of size
$3$ (see Figure \ref{fig:collar-pendant}). The $1$-shell gives
rise to the zero eigenvalue, and the composition of the kernels of
$B,B^{2},B^{3}$ is similar to that of (a). The pendant gives rise
to three new eigenvalues that are all third roots of unity. Example
eigenvectors are shown. Figure \ref{fig:tree-circle}(c) shows a graph
with two cycles and non-empty $1$-shell. The $1$-shell is the same
as in (a) and (b), and the cycle is a collar of length $4$. As before,
the $1$-shell gives rise to the zero eigenvalue and the kernels of
$B,B^{2},B^{3}$. The collar gives rise to eigenvalues that are fourth
roots of unity. Example eigenvectors are shown. The multiplicities
of the roots of unity in (b) and (c) is given by Theorem \ref{thm:unit-multiplicity}.

Figure \ref{fig:karate}(a) shows the well-known Karate Club graph
\cite{zachary1977information}. Its $1$-shell is comprised of only
one layer with one node (the orange node in the Figure), and therefore
$AM(0)=2,GM(0)=1$. The purple nodes form a collar of size $4$. Any
four of the green nodes that form a cycle form a collar of size $4$.
Note the nodes of degree greater than $2$ in the purple collar are
not neighbors, while the nodes of degree greater than $2$ in the
green collars are neighbors.

\begin{figure}
\includegraphics{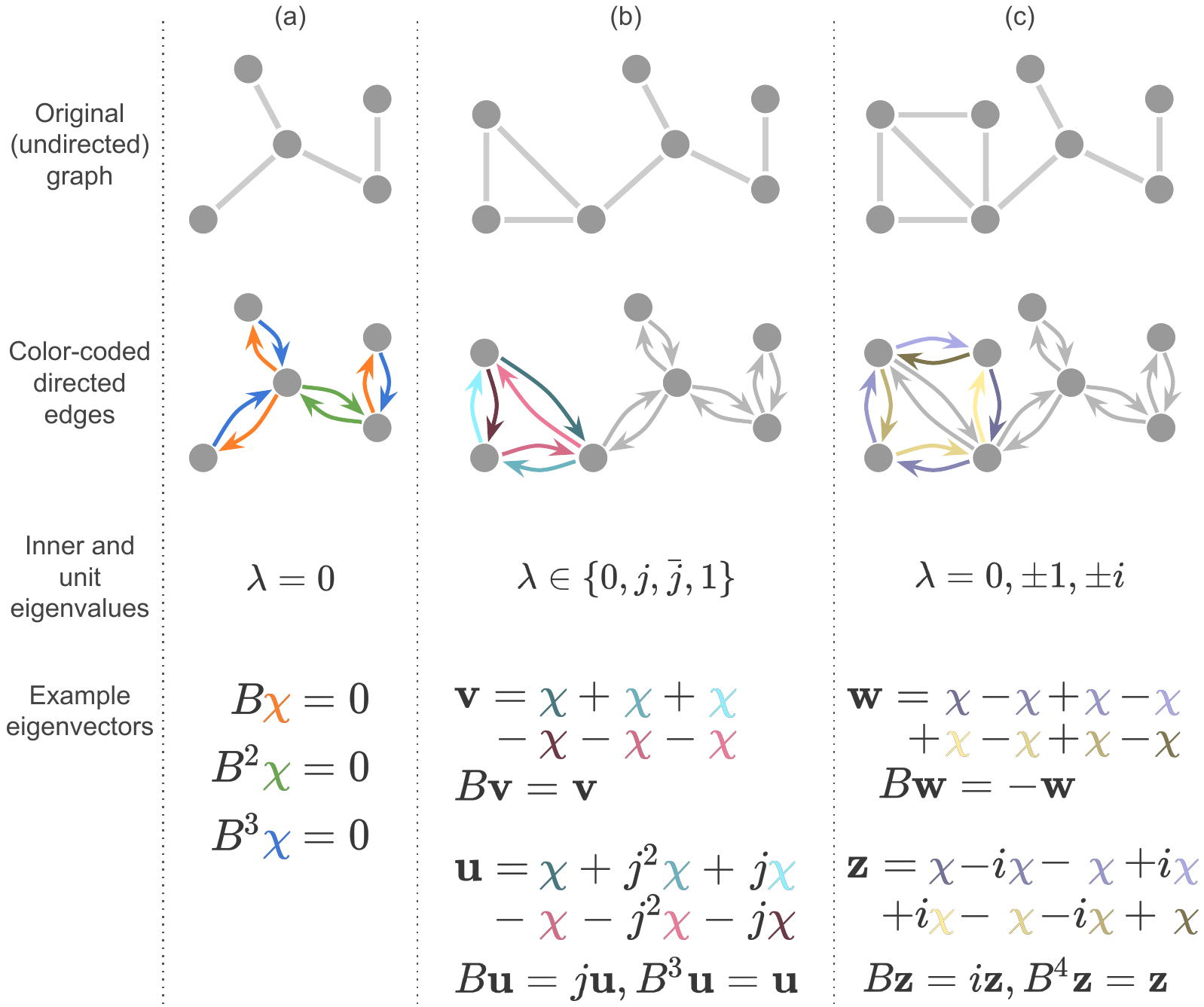}

\caption{\label{fig:tree-circle}Inner and unit eigenvalues of example graphs.
See Section \ref{sec:examples-tree-circle} for discussion. Here we
have $i^{2}=-1$ and $j=\frac{-1+i\sqrt{3}}{2}$. The characteristic
vectors $\chi$ are color-coded. For example, $\color[HTML]{ff7f28}\chi$
represents the characteristic vector of any of the orange edges, while
$\color[HTML]{477880}\chi$ is the characteristic vector of the sole
edge of the same color. }
\end{figure}

\begin{figure}
\begin{centering}
\includegraphics[scale=0.8]{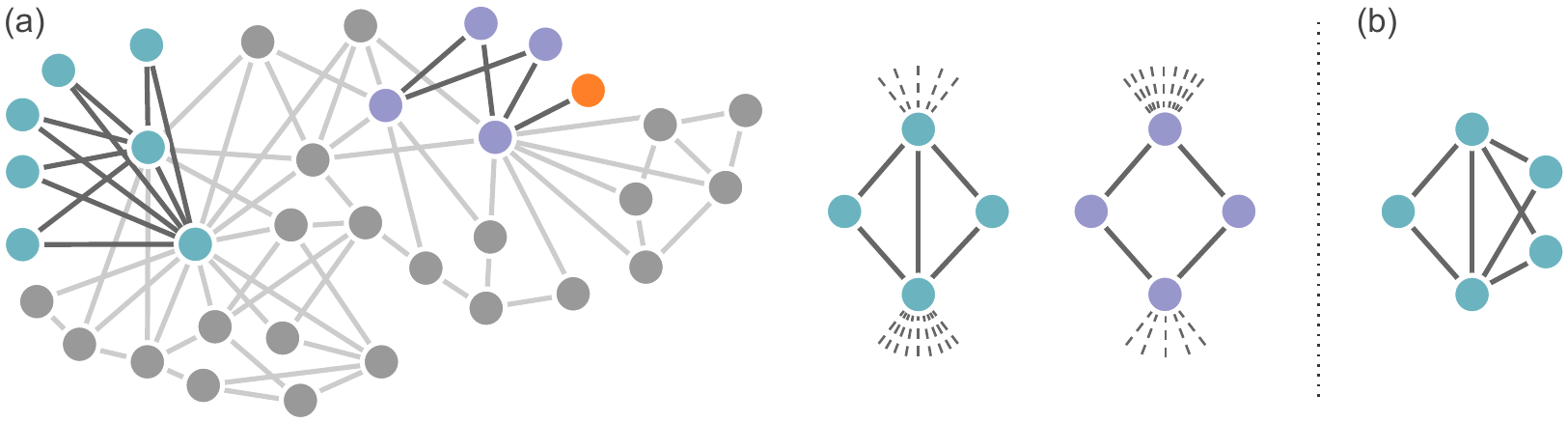}
\par\end{centering}
\caption{\label{fig:karate}\textbf{(a) }The Karate Club graph. See Section
\ref{sec:examples-tree-circle} for discussion. \textbf{(b)} A graph
made of two linearly independent and overlapping collars. See Section
\ref{sec:examples-unit} for discussion. Note it is a subgraph of
the Karate Club graph.}
\end{figure}

\section{Graphs with two cycles or more\label{sec:2-cycles}}

In this Section, all graphs have at least two cycles, and we continue
to assume minimum md$2$. We analyze the eigenvalues in order of increasing
magnitude, following the categories shown in Figure \ref{fig:eigenvalues-by-magnitude}.
The case $\lambda=0$ has already been dealt with in Section \ref{sec:1-shell}.
We recall well-known results on the impossibility of finding eigenvalues
with $0<|\lambda|<1$ in Section \ref{sec:inner}, which completes
the characterization of the inner eigenvalues. Next, we treat the
unit eigenvalues, $|\lambda|=1$, case by case in Section \ref{sec:unit}.
We then focus on the outer eigenvalues in Section \ref{sec:outer},
where we formulate a conjecture on their simplicity. Finally, we recall
known results on leading eigenvalues i.e. those with $|\lambda|=\rho$. 

\subsection{The inner eigenvalues\label{sec:inner}}

It is a well-known fact that eigenvalues with $0<|\lambda|<1$ are
in fact impossible. Kotani and Sunada \cite{Kotani2000}, Theorem
1.3(a), prove this in the language of Zeta functions, by making use
of the Ihara-Bass formula (\ref{eqn:ihara-bass}). For completeness,
here we paraphrase their theorem in the language of the NB-matrix.
\begin{thm}
[from \cite{Kotani2000}] Let $G$ be a graph with minimum degree
at least $2$ with at least two cycles. Then, every NB-eigenvalue
$\lambda$ satisfies $1\leq|\lambda|$. \qed
\end{thm}

\begin{rem*}
The proof of this Theorem can be found in Section 6 of \cite{Kotani2000}.
It can be read without much background in the theory of graph Zeta
functions, by keeping in mind that if $\lambda$ is a NB-eigenvalue
then $1/\lambda$ is a pole of (\ref{eqn:ihara-bass}).
\end{rem*}

\subsection{The unit eigenvalues\label{sec:unit}}

We give a complete characterization of the unit eigenvalues and their
eigenvectors in arbitrary graphs. Some of our arguments require the
graph to be md$2$, as we have been assuming, but Section \ref{sec:zero-one}
establishes that nodes of degree $1$ (and in fact any node in the
$1$-shell) have no influence on the unit eigenvalues. Thus, the results
here are valid for arbitrary graphs, without restriction. We first
prove that all unit eigenvalues must be roots of unity. In this case,
there exists a set of nodes $\mathcal{C}$ that is always a pendant,
a collar, or a bracelet such that the associated eigenvector $\mathbf{v}$
is supported on $\mathcal{C}$, i.e. $\mathbf{v}_{k\to l}\neq0$ if
and only if $k,l\in\mathcal{C}$. 

\subsubsection{Only roots of unity are NB-eigenvalues}

Assume $B\mathbf{v}=\lambda\mathbf{v}$. By the properties of unitary
matrices, $\lambda$ is unitary if and only if$B^{*}B\mathbf{v}=BB^{*}\mathbf{v}=\mathbf{v}$.
Therefore, we start our discussion by computing $B^{*}B$ and $BB^{*}$.
For this purpose, define $\nbcent[l]\coloneqq\sum_{i}a_{il}\mathbf{v}_{i\to l}$
and $\incent\coloneqq\sum_{i}a_{ik}\mathbf{v}_{k\to i}$. Recall from
Equation (\ref{eqn:nb-centrality}) that if $\mathbf{v}$ is the Perron
eigenvector of $B$, then $\nbcent$ is the NB-centrality of $k$.
\begin{lem}
\label{lem:bbt}For any vector $\mathbf{v}$, the following hold (see
Figure \ref{fig:bbt}).
\begin{align}
\left(B^{*}B\mathbf{v}\right)_{k\to l} & =\left(d_{l}-2\right)\nbcent[l]+\mathbf{v}_{k\to l}\\
\left(BB^{*}\mathbf{v}\right)_{k\to l} & =\left(d_{k}-2\right)\incent+\mathbf{v}_{k\to l}.
\end{align}
\end{lem}

\begin{proof}
This is direct from the definition of $B$. For brevity, we show only
the case $B^{*}B\mathbf{v}$. 
\begin{align*}
\left(B^{*}B\mathbf{v}\right)_{k\to l} & =\sum_{i\to j}\delta_{il}\left(1-\delta_{jk}\right)\sum_{r\to s}\delta_{is}\left(1-\delta_{jr}\right)\mathbf{v}_{r\to s}\\
 & =\sum_{j\neq k}a_{jl}\left(\sum_{r}a_{rl}\mathbf{v}_{r\to l}-\mathbf{v}_{j\to l}\right)\\
 & =\left(d_{l}-1\right)\nbcent[l]-\nbcent[l]+\mathbf{v}_{k\to l}.
\end{align*}
\end{proof}

\begin{figure}
\begin{centering}
\includegraphics{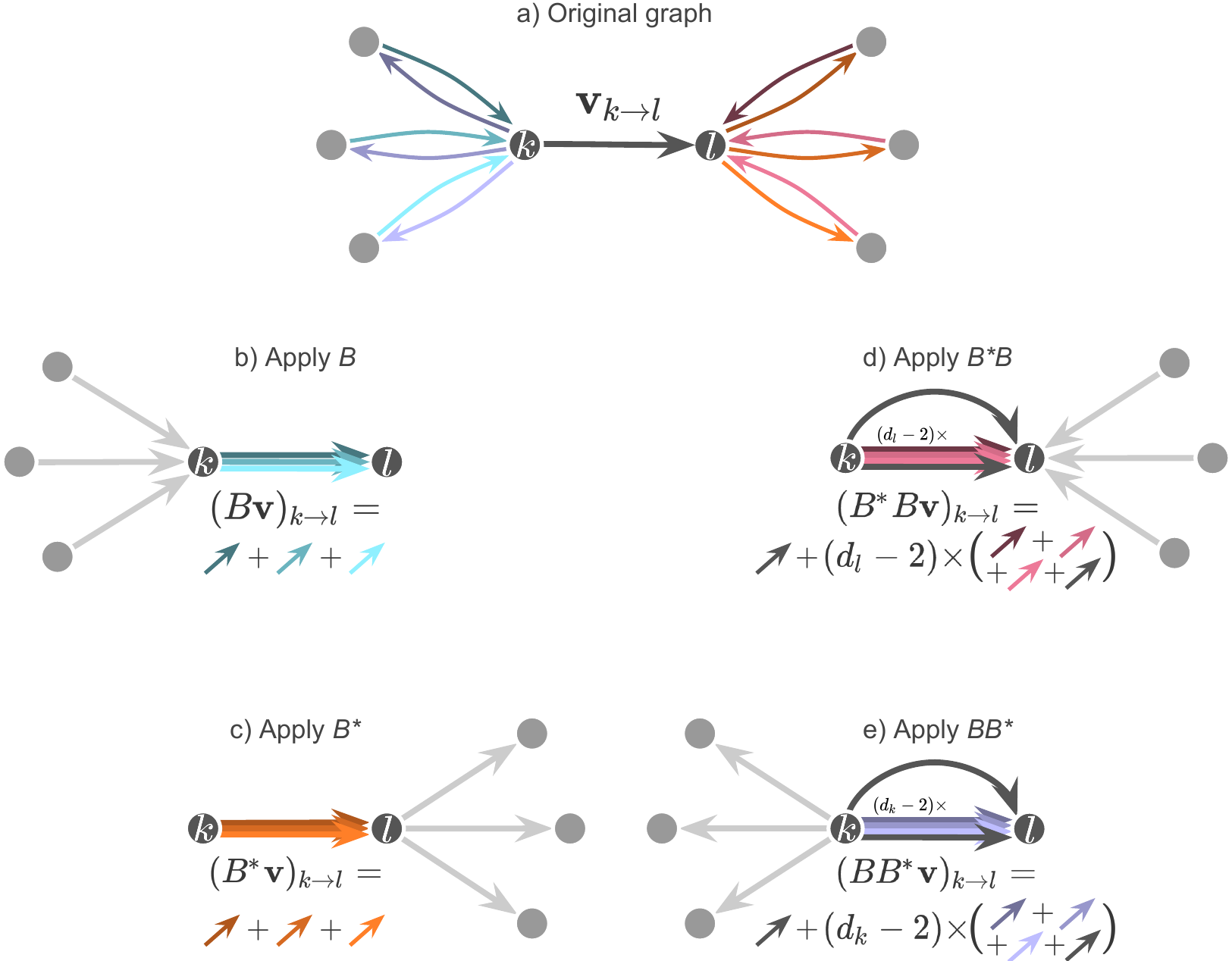}
\par\end{centering}
\caption{\label{fig:bbt}The action of $B,B^{*},B^{*}B$ and $BB^{*}$ on a vector $\mathbf{v}$. See Lemma \ref{lem:bbt}.}
\end{figure}

\begin{rem*}
Note that $\incent$ sums over the directed edges that have $k$ as
a source, reflected by the use of ``$k\!\to$'' in the notation.
Similarly, $\nbcent$ sums over the directed edges that have $k$
as a target, reflected by the use of ``$\to\!k$''. We pronounce
$\incent$ as \emph{``$\mathbf{v}$ from k'' }and $\nbcent$ as\emph{
``$\mathbf{v}$ into $k$''.}
\end{rem*}
\begin{example}
\begin{figure}
\begin{centering}
\includegraphics{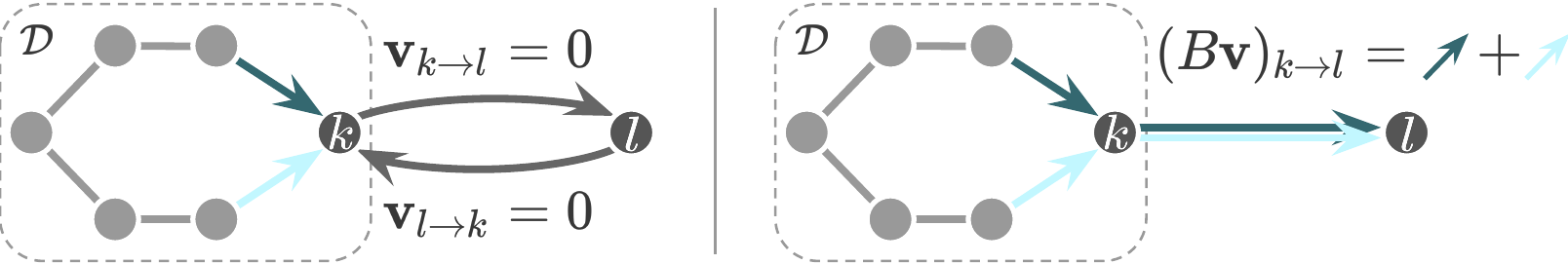}
\par\end{centering}
\caption{\label{fig:cycle-support}The nodes in $\mathcal{D}$ induce a cycle.
The node $k\in\mathcal{D}$ has a neighbor $l\protect\notin\mathcal{D}$.\textbf{
Left:} a vector $\mathbf{v}$ supported on $\mathcal{D}$. \textbf{Right:}
if $\mathbf{v}$ is an eigenvector, the sum of all values incoming
to $k$ must be zero. If $\left(B\mathbf{v}\right)_{k\to l}$ is non-zero,
we say that $\mathbf{v}$ leaks out of $\mathcal{D}$ via $k$. Nodes
with degree $2$ can never leak.}
\end{figure}

\label{exm:cycle-support}Due to Lemma \ref{lem:bbt}, to understand
the eigenvectors of unit eigenvalues, it is sufficient to understand
those $\mathbf{v}$ that satisfy $\left(d_{k}-2\right)\incent=\left(d_{l}-2\right)\nbcent[l]=0$
for each pair of neighboring $k,l$. For this purpose, consider the
following situation and the accompanying Figure \ref{fig:cycle-support}.
Let $G$ be a graph with NB-matrix $B$. Let $\mathcal{D}$ be a set
of nodes whose induced subgraph is a cycle. Suppose $B\mathbf{v}=\lambda\mathbf{v}$
with $\lambda\neq0$ and that $\mathbf{v}$ is supported on $\mathcal{D}$.
Let $k\in\text{\ensuremath{\mathcal{D}}},l\notin\mathcal{D}$ be neighbors.
Since $\mathbf{v}$ is an eigenvector supported on $\mathcal{D}$,
we have
\begin{equation}
0=\lambda\mathbf{v}_{k\to l}=\left(B\mathbf{v}\right)_{k\to l}=\sum_{i}a_{ik}\mathbf{v}_{i\to k}-\mathbf{v}_{l\to k}=\nbcent,
\end{equation}
where the last equality uses the fact that $\mathbf{v}_{l\to k}=0$.
Thus, a necessary condition for $\mathbf{v}$ to be an eigenvector
supported on a cycle $\mathcal{D}$ is that for every $k\in\mathcal{D}$
with a neighbor outside of $\mathcal{D}$, we must have $\nbcent=0$.
Note that if $k\in\mathcal{D}$ has no neighbors outside of $\mathcal{D}$,
i.e. if its degree is $2$, then there is no restriction on $\nbcent$.
In other words, $\mathbf{v}$ satisfies $\left(d_{k}-2\right)\nbcent=0$
for each $k$, and therefore $B^{*}B\mathbf{v}=\mathbf{v}$ by Lemma
\ref{lem:bbt}. Furthermore, we have $BB^{*}\mathbf{v}=\mathbf{v}$
as well, by Lemma \ref{lem:nb-in-cent}. Lastly, if $r$ is the length
of the cycle induced by $\mathcal{D}$, we have $B^{r}\mathbf{v}=\mathbf{v}$
and thus $\lambda^{r}=1$. 
\end{example}

Before moving forward, let us capture the property $\left(d_{k}-2\right)\nbcent=0$
with the following terminology, inspired by Example \ref{exm:cycle-support}.
\begin{defn}
\label{def:leaky}Consider a vector $\mathbf{v}$ (not necessarily
an eigenvector) with support $\mathcal{D}$ (not necessarily a cycle).
If there is a $k\in\mathcal{D}$ such that $\left(d_{k}-2\right)\nbcent\neq0$,
we say that \emph{$\mathbf{v}$ leaks out of $\mathcal{D}$ via $k$},
or simply that \emph{$\mathbf{v}$ is leaky}. If $\mathbf{v}$ does
not leak via any node, we say that $\mathbf{v}$ is \emph{non-leaky}.
See Figure \ref{fig:cycle-support}.
\end{defn}

Example \ref{exm:cycle-support} shows that if $\mathbf{v}$ is an
eigenvector supported on a cycle, then $\mathbf{v}$ must be non-leaky
and its corresponding eigenvalue must be a root of unity. The following
theorem is essentially a generalization of this observation.
\begin{thm}
\label{thm:root}Suppose $B\mathbf{v}=\lambda\mathbf{v}$ with $\lambda\neq0$.
$\mathbf{v}$ is non-leaky if and only if $\lambda$ is a root of
unity.
\end{thm}

\begin{proof}
If $\lambda$ is a root of unity then it is unitary; by Lemma \ref{lem:non-leaky-unitary},
$\mathbf{v}$ is non-leaky. Now assume $\mathbf{v}$ is non-leaky
and let $\mathcal{D}$ be the support of $\mathbf{v}$. We proceed
in two cases.
\begin{enumerate}
\item Assume that $\mathcal{D}$ contains no nodes of degree $2$ and take
two neighbors $k,l$ with $\mathbf{v}_{k\to l}\neq0$. Since $\mathbf{v}$
is non-leaky, we must have $\nbcent=\nbcent[l]=0$. By Equation (\ref{eqn:eigenvector-in-two-directions})
this is equivalent to $\lambda\mathbf{v}_{k\to l}+\mathbf{v}_{l\to k}=0=\lambda\mathbf{v}_{l\to k}+\mathbf{v}_{k\to l}$.
Multiply the first equation by $\lambda$ and replace in the second
equation to obtain $0=\left(\lambda^{2}-1\right)\mathbf{v}_{k\to l}.$
Therefore, $\lambda$ must be a square root of unity.
\item \label{enu:thm-degree-2}Assume that $k\in\mathcal{D}$ has degree
$2$. We will show there exists a vector $\mathbf{c}$ such that $B\mathbf{c}=\lambda\mathbf{c}$
and $\mathbf{c}$ is non-leaky and supported on a cycle. In that case,
$\lambda$ must be a root of unity by Example \ref{exm:cycle-support}.
Let $i,l\in\mathcal{D}$ be the two neighbors of $k$, and note that
$\lambda\text{\ensuremath{\mathbf{v}}}_{k\to l}=\mathbf{v}_{i\to k}$.
Take a $\mathcal{C}\subset\mathcal{D}$ such that $i,k,l\in\mathcal{C}$
and the graph induced by $\mathcal{C}$ is a cycle. This is always
possible due to Lemma \ref{lem:cycle-degree-2}. Suppose $\mathcal{C}$
contains $r$ nodes and label them by consecutive numbers $k=1,l=2,\ldots,i=r$.
Define $\mathbf{c}_{i\to k}\coloneqq\mathbf{v}_{i\to k}$ and $\mathbf{c}_{k\to l}\coloneqq\mathbf{v}_{k\to l}$.
Define all other edges as $\mathbf{c}_{j\to(j+1)}\coloneqq\lambda\mathbf{c}_{(j-1)\to j},\,j=2,\ldots,r-1$.
By construction, $\mathbf{c}$ is non-leaky and supported on a cycle;
by Example (\ref{exm:cycle-support}) it must be an eigenvector with
eigenvalue $\lambda$. Therefore, $\lambda^{r}=1$.
\end{enumerate}
\end{proof}
\begin{thm}
\label{thm:unit-multiplicity}$\lambda$ is not defective, i.e. $AM(\lambda)=GM(\lambda)$.
\end{thm}

\begin{proof}
We show that any generalized eigenvector must be an eigenvector. Let
$\mathbf{c}$ be such that $\left(B-\lambda I\right)^{2}\mathbf{c}=0$,
and define $\mathbf{v}\coloneqq\left(B-\lambda I\right)\mathbf{c}$.
Note that $\mathbf{v}$ is an eigenvector of eigenvalue $\lambda$
(or it is the zero vector). We first prove that $\mathbf{c}$ must
be non-leaky; we proceed in three cases. First, if $\mathbf{v}_{k\to l}$
equals $0$, we have $\left(B\mathbf{c}\right)_{k\to l}=\lambda\mathbf{c}_{k\to l}$.
That is, $\mathbf{c}$ behaves like an eigenvector outside the support
of $\mathbf{v}$. In particular, {\ensuremath{\left( d_k -2 \right) \tensor*[^{k}]{\vec{\mathbf{c}}}{}} = 0}
for any $k$ not in the support of $\mathbf{v}$. Second, for any
node $k$ of degree $2$, we have ${\ensuremath{\left(d_{k}-2\right)\tensor*[^{k}]{\vec{\mathbf{c}}}{}}=0}$,
regardless of whether or not $k$ is in the support of $\mathbf{v}$. 

Third, let $k$ be in the support of $\mathbf{v}$ with $d_{k}>2$
and thus $\incent=0$. Using the definition $\mathbf{v}_{k\to l}={\ensuremath{\tensor*[^{k}]{\vec{\mathbf{c}}}{}}}-\mathbf{c}_{l\to k}-\lambda\,\mathbf{c}_{k\to l}$
and summing over every neighbor $l$ of $k$, we obtain $\lambda{\ensuremath{\tensor*[^{k}]{\vec{\mathbf{c}}}{}}}=\left(d_{k}-1\right){\ensuremath{\tensor*[]{\vec{\mathbf{c}}}{^{k}}}}$,
or equivalently 
\begin{equation}
\lambda{\ensuremath{\tensor*[^{k}]{\vec{\mathbf{c}}}{}}}-{\ensuremath{\tensor*[]{\vec{\mathbf{c}}}{^{k}}}}=\left(d_{k}-2\right){\ensuremath{\tensor*[]{\vec{\mathbf{c}}}{^{k}}}}.\label{eqn:target-not-leaky}
\end{equation}
In the following, we show that $\mathbf{c}$ does not leak via $k$
by showing that the two members of this last equation in fact equal
zero. We proceed in two sub-cases.
\begin{enumerate}
\item Assume $\lambda^{2}=1$. Let $l$ be a node outside the support of
$\mathbf{v}$, i.e. $\mathbf{c}$ behaves like an eigenvector on $k\to l$
and 
\[
\lambda\mathbf{c}_{k\to l}=\left(B\mathbf{c}\right)_{k\to l}={\ensuremath{\tensor*[]{\vec{\mathbf{c}}}{^{k}}}}-\mathbf{c}_{l\to k}=\sum_{l\in supp(\mathbf{v})}\mathbf{c}_{l\to k}+\sum_{l\notin supp(\mathbf{v})}\mathbf{c}_{l\to k}-\mathbf{c}_{l\to k}.
\]

Note that $\sum_{l\notin supp(\mathbf{v})}\mathbf{c}_{l\to k}=0$
since $\mathbf{c}$ is an eigenvector outside of the support of $\mathbf{v}$
and therefore it does not leak through $k$. Therefore $\lambda\mathbf{c}_{k\to l}=\sum_{l\in supp(\mathbf{v})}\mathbf{c}_{l\to k}-\mathbf{c}_{l\to k}$.
On the other hand, we have $\mathbf{c}_{k\to l}+\lambda\mathbf{c}_{l\to k}={\ensuremath{\tensor*[]{\vec{\mathbf{c}}}{^{l}}}}=0$,
since $\mathbf{c}$ does not leak through $l$. These two equations
simplify to 
\[
\sum_{l\in supp(\mathbf{v})}\mathbf{c}_{l\to k}=\left(1-\lambda^{2}\right)\mathbf{c}_{l\to k}=0.
\]

All together, we have 
\[
{\ensuremath{\tensor*[]{\vec{\mathbf{c}}}{^{k}}}}=\sum_{l\in supp(\mathbf{v})}\mathbf{c}_{l\to k}+\sum_{l\notin supp(\mathbf{v})}\mathbf{c}_{l\to k}=0+0=0,
\]

and Equation (\ref{eqn:target-not-leaky}) equals zero, as desired.
\item Assume $\lambda^{r}=1$, with $r\neq2$. To fix ideas, assume that
$\mathbf{v}$ is supported on a single cycle. In this case, $k$ has
exactly two neighbors in the support of $\mathbf{v}$, call them $i$
and $j$. Since $\mathbf{v}$ is non-leaky, we have 
\begin{alignat*}{1}
\incent & =0\\
\mathbf{v}_{k\to j}+\mathbf{v}_{k\to i} & =0\\
{\ensuremath{\tensor*[]{\vec{\mathbf{c}}}{^{k}}}}-\mathbf{c}_{j\to k}-\lambda\mathbf{c}_{k\to j}+{\ensuremath{\tensor*[]{\vec{\mathbf{c}}}{^{k}}}}-\mathbf{c}_{i\to k}-\lambda\mathbf{c}_{k\to i} & =0\\
\sum_{l\notin supp(\mathbf{v})}\mathbf{c}_{l\to k}+\mathbf{c}_{i\to k}+\sum_{l\notin supp(\mathbf{v})}\mathbf{c}_{l\to k}+\mathbf{c}_{j\to k} & =\lambda\left(\mathbf{c}_{k\to i}+\mathbf{c}_{k\to j}\right)\\
\mathbf{c}_{i\to k}+\mathbf{c}_{j\to k} & =\lambda\left(\mathbf{c}_{k\to i}+\mathbf{c}_{k\to j}\right)\\
\mathbf{c}_{i\to k}+\mathbf{c}_{j\to k}+\sum_{l\notin supp(\mathbf{v})}\mathbf{c}_{l\to k} & =\lambda\left(\mathbf{c}_{k\to i}+\mathbf{c}_{k\to j}+\sum_{l\notin supp(\mathbf{v})}\mathbf{c}_{k\to l}\right)\\
{\ensuremath{\tensor*[]{\vec{\mathbf{c}}}{^{k}}}} & =\lambda{\ensuremath{\tensor*[^{k}]{\vec{\mathbf{c}}}{}}},
\end{alignat*}

where we have used that $\sum_{l\notin supp(\mathbf{v})}\mathbf{c}_{l\to k}=\sum_{l\notin supp(\mathbf{v})}\mathbf{c}_{k\to l}=0$.
This shows that Equation (\ref{eqn:target-not-leaky}) equals zero.
The general case when $\mathbf{v}$ is not supported on a single cycle
is similar but taking into consideration that $k$ has exactly two
neighbors in each of the cycles on which $\mathbf{v}$ is supported.
\end{enumerate}
We have established that $\mathbf{c}$ is non-leaky. Now write $\mathbf{c}=\mathbf{c}'+\mathbf{c}''$,
where $\mathbf{c}'$ is supported on the same support as $\mathbf{v}$,
and $\mathbf{c}''$ is supported outside of it. As per our previous
observation, $\mathbf{c}''$ is an eigenvector of eigenvalue $\lambda$
and therefore it is non-leaky. Since $\mathbf{c}$ is also non-leaky,
$\mathbf{c}'$ must be non-leaky as well. Per Lemma \ref{lem:non-leaky-combination},
$\mathbf{c}'$ must be the linear combination of eigenvectors. All
of these must correspond to the same eigenvalue $\lambda$ as otherwise,
$\mathbf{v}$ would not be in the kernel of $\left(B-\lambda I\right)$.
We have proved that both $\mathbf{c}'$ and $\mathbf{c}''$ are eigenvectors
of $\lambda$, and thus $\mathbf{c}$ is as well and $\mathbf{v}$
was the zero vector all along.
\end{proof}
We have proved that the only numbers on the unit circle that may be
NB-eigenvalues are the roots of unity, and when they are, they are
never defective. We proceed to compute the exact multiplicity of the
complex roots of unity and real roots of unity in turn.

\subsubsection{\label{sec:complex-roots}Complex roots of unity}

Theorem \ref{thm:root} shows that, in graphs with no nodes of degree
$2$, only $\pm1$ may be unit NB-eigenvalues. In graphs that do have
complex roots of unity, we have the following characterization. In
this section, we fix a nonzero $\lambda$ and let $B\mathbf{v}=\lambda\mathbf{v}$
with $\lambda^{r}=1$ but $\lambda^{2}\neq1$.
\begin{prop}
$\mathbf{v}$ can we written as $\mathbf{v}=\sum_{i=1}^{t}\mathbf{c}^{i}$,
where each $\mathbf{c}^{i}$ is an eigenvector supported on a different
cycle.
\end{prop}

\begin{proof}
With the notations used in Theorem \ref{thm:root}, put $\mathbf{v}^{0}\coloneqq\mathbf{v}$
and $\mathbf{c}^{1}\coloneqq\mathbf{c}$. Define $\mathbf{v}^{1}\coloneqq\mathbf{v}^{0}-\mathbf{c}^{1}$.
Since both $\mathbf{v}^{0}$ and $\mathbf{c}^{1}$ are non-leaky eigenvectors,
$\mathbf{v}^{1}$ is a non-leaky eigenvector as well. Let $\mathcal{D}_{1}$
be the support of $\mathbf{v}^{1}$. By construction, we have $\mathbf{v}_{i\to k}^{1}=\mathbf{v}_{k\to l}^{1}=0$
and therefore $k\notin\mathcal{D}_{1}$. But since $\lambda^{2}\neq1$,
there must be a $k_{2}\in\mathcal{D}_{1}$ with degree $2$. Thus
we can construct another $\mathbf{c}^{2}$ supported on a cycle containing
$k_{2}$ and define $\mathbf{v}^{2}\coloneqq\mathbf{v}^{1}-\mathbf{c}^{2}$.
Note that the support of $\mathbf{v}^{2}$ is a proper subset of $\mathcal{D}_{1}$
as it does not contain $k_{2}$. We can iterate this construction
for $t$ steps until support $\mathbf{v}^{t}$ is supported on a single
cycle, i.e. until $\mathbf{v}^{t}=\mathbf{c}^{t}$.
\end{proof}
\begin{prop}
Let $\mathcal{C}$ be a set of $r$ nodes that induce either a cycle
or a figure eight graph, and suppose $\lambda^{r}=1$ but $\lambda^{2}\neq1$.
Assume there exists an eigenvector supported on the edges in the graph
induced by $\mathcal{C}$. Then there is only one such eigenvector,
up to a scalar.
\end{prop}

\begin{proof}
Let $B\mathbf{v}=\lambda\mathbf{v}$, where $\mathbf{v}$ is nonzero
within the graph induced by $\mathcal{C}$ and zero outside of it,
and thus $\mathbf{v}$ has $2r$ nonzero coordinates. It is sufficient
to show that the condition of being a non-leaky eigenvector supported
on $\mathcal{C}$ determines a system of $2r-1$ equations. Label
the nodes of $\mathcal{C}$ by $0,2,\ldots,r-1$ such that the node
$i$ is adjacent to the nodes labeled $i-1$ and $i+1$; here labels
are taken $\mod r$. Since $\mathbf{v}$ is an eigenvector, we have
$\mathbf{v}_{i\to(i+1)}=\lambda\mathbf{v}_{(i-1)\to i}$ and $\mathbf{v}_{\left(i+1\right)\to i}=\lambda\mathbf{v}_{i\to\left(i-1\right)}$
for each $i$; this gives $2r-2$ independent equations. Without loss
of generality we may assume that the node with label $0$ has degree
greater than $2$. Since $\mathbf{v}$ does not leak through the node
with label $0$, we have $0=\nbcent[0]=\mathbf{v}_{r-1\to0}+\mathbf{v}_{1\to0}$,
which is an equation independent of the others, for a total of $2r-1$
equations, completing the proof. (Note that in a cycle graph, the
condition of non-leakiness is trivial as all nodes have degree $2$.
In that case, we only have a system with $2r-2$ equations, whence
the geometric multiplicity of $\lambda$ is $2$; cf. Section \ref{sec:1-cycle}.) 
\end{proof}
\begin{prop}
If $r$ is odd, $\mathcal{C}$ must be a pendant. If $r$ is even,
$\mathcal{C}$ must be a collar or a bracelet.
\end{prop}

\begin{proof}
Suppose the node with label $0$ has degree larger than $2$ and suppose
$\mathbf{v}_{r-1\to0}=1$, which fixes all other coordinates to be
$\mathbf{v}_{i\to\left(i+1\right)}=\lambda^{-i}$ and $\mathbf{v}_{\left(i+1\right)\to i}=-\lambda^{r+1-i}$.
It suffices to inspect $\nbcent[i]=\mathbf{v}_{\left(i-1\right)\to i}+\mathbf{v}_{\left(i+1\right)\to i}=\lambda^{-\left(i-1\right)}-\lambda^{r+1-i}$
for each $i$; if $\nbcent[i]\neq0$ then $i$ must have degree $2$.
The properties of sums of roots of unity are well-known. In particular,
if $r$ is odd, $\nbcent[i]$ is zero only when $i=0$. In other words,
only the node with label $0$ can have degree greater than $2$, which
means that $\mathcal{C}$ is a pendant. If $r$ is even, $\nbcent[i]$
is zero only when $i=0$ or $i=r/2$. In this case, and if $\mathcal{C}$
induces a cycle, then it is a collar; if it induces a figure eight
graph (and the nodes $0$ and $r/2$ are actually the same), it is
a bracelet.
\end{proof}
\begin{cor}
\label{cor:eigenspace-basis}$GM(\lambda)$ equals the number of pendants
or collars or bracelets of length $r$. Equivalently, the eigenspace
corresponding to $\lambda$ has a basis $\{\mathbf{c}^{i}\}_{i=1}^{t}$
such that the support of each $\mathbf{c}^{i}$ is a pendant or a
collar or a bracelet. \qed
\end{cor}

\subsubsection{\label{sec:real-roots}Real roots of unity}

We proceed to find the algebraic and geometric multiplicity of $\lambda=\pm1$.
Theorem (\ref{thm:unit-multiplicity}) establishes that these quantities
are equal, though we present different proofs for each. The proofs
for computing the algebraic multiplicity are related to the Ihara-Bass
formula of Equation (\ref{eqn:ihara-bass}), while the proofs for
computing the geometric multiplicities yield a basis for the corresponding
eigenspace similar to that exhibited for complex roots of unity in
Corollary \ref{cor:eigenspace-basis}; see Corollary \ref{cor:eigenspace-basis-real}.

A graph $G$ has at least two cycles if and only if it has more edges
than nodes: $m>n$. In this case, the Ihara-Bass formula (\ref{eqn:ihara-bass})
immediately implies that $AM(\pm1)\geq m-n$. For this Section, recall
that $D-A$ is called the \emph{Laplacian }matrix of $G$, which is
always singular, and whose rank is $n-1$ if and only if the graph
is connected. An argument closely related to the following proof can
be found in \cite{northshield1998,Grindrod2018}, though we have arrived
at it independently.
\begin{prop}
\label{pro:am-plus1}Let $G$ have at least two cycles, i.e. $m>n$.
Then $AM(1)=m-n+1$.
\end{prop}

\begin{proof}
Define $f(u)\coloneqq\det\left(I-uA+u^{2}(D-I)\right)$ and observe
that $f(1)=\det\left(D-A\right)=0$. Therefore, $f(u)=(u-1)g(u)$
and the Ihara-Bass formula (\ref{eqn:ihara-bass}) implies $AM(1)\geq m-n+1$.
Showing $g(1)\neq0$ finishes the proof. First, note that $g(1)$
equals $f'(1)$. The so-called Jacobi formula shows $f'(1)=\Tr\left(\adj\left(D-A\right)\left(D-A+D-2I\right)\right)$
(see \cite{petersen2012matrix}, Equation (41)). Further, well-known
properties of the adjugate show that $\adj\left(D-A\right) = \eta \mathbf{1}\mathbf{1}^{T}$ for some nonzero $\eta$
(see \cite{horn2012matrix}, Section 0.8.2 and \cite{merris1994laplacian}). All together, we have
\begin{alignat}{1}
g(1)=f'(1) & =\Tr\left(\adj\left(D-A\right)\left(D-A+D-2I\right)\right)\label{eqn:f-prime-one}\\
 & =\eta \Tr\left(\mathbf{1}\mathbf{1}^{T}\left(D-A\right)+\mathbf{11}^{T}\left(D-2I\right)\right)\nonumber \\
 & =\eta \mathbf{1}^{T}\left(D-A\right)\mathbf{1} + \eta \mathbf{1}^{T}\left(D-2I\right)\mathbf{1}\nonumber \\
 & = \eta \left( 2m-2n \right) \neq 0,
\end{alignat}
where the third line uses Lemma \ref{lem:trace-cyclic}.
\end{proof}
\begin{prop}
\label{pro:gm-plus1}Let $G$ have at least two cycles. Then $GM(1)=m-n+1$.
\end{prop}

\begin{proof}
Since $GM$ is bounded above by $AM$, we have $GM(1)\leq m-n+1$.
Thus we only need to show that there exists a set of $m-n+1$ linearly
independent vectors that satisfy $B\mathbf{v}=\mathbf{v}$. Inspection
of (\ref{eqn:eigenvector-in-two-directions}) when $\lambda=1$ shows
that there exists a global constant $q$ such that 
\[
\mathbf{v}_{k\to l}+\mathbf{v}_{l\to k}=q=\nbcent[l]
\]
for any neighboring nodes $k,l$. Summing the left equation for each
edge yields 
\[
\frac{1}{2}\sum_{k,l}a_{kl}\left(\mathbf{v}_{k\to l}+\mathbf{v}_{l\to k}\right)=mq,
\]
while summing the right equation for each of the $n$ nodes yields
\[
nq=\sum_{l}\nbcent[l].
\]
Note these two equations sum each of the coordinates of $\mathbf{v}$
exactly once; thus $nq=mq$ and $q=0$. We conclude that $\mathbf{v}$
satisfies the system 
\begin{equation}
\begin{cases}
\nbcent[l]=0 & \text{for each node }l,\\
\mathbf{v}_{k\to l}+\mathbf{v}_{l\to k}=0 & \text{for each edge }k-l.
\end{cases}\label{eqn:system-1}
\end{equation}
Now take a spanning tree $T$ of $G$ and an edge $u_{0}-v_{0}$ not
in $T$. The edge $u_{0}-v_{0}$ determines a unique NB-cycle $c$
all of whose edges are in $T$ except for $u_{0}-v_{0}$. Choose an
arbitrary orientation for the cycle, say $c=u_{0}\to v_{0},u_{1}\to v_{1},\ldots,u_{r}\to v_{r}=u_{0}$
and consider the vector $\mathbf{v}^{u_{0}\to v_{0}}\coloneqq\sum_{i=0}^{r}\chi^{u_{i}\to v_{i}}-\chi^{v_{i}\to u_{i}}$.
It can be manually checked that $\mathbf{v}$ satisfies (\ref{eqn:system-1}).
(See Figure \ref{fig:tree-circle}(b) for an example.) Now, for each
$\left(k-l\right)\notin T$, define $\mathbf{v}^{k\to l}$ similarly
to $\mathbf{v}^{u_{0}\to v_{0}}$ above. The set $\{\mathbf{v}^{k\to l}:\left(k-l\right)\notin T\}$
is linearly independent since each vector $\mathbf{v}^{k\to l}$ has
a non-zero entry at coordinate $k\to l$, and all other vectors are
zero at that coordinate. Since there are exactly $m-n+1$ edges not
in $T$, we have proved $GM(1)\geq m-n+1$. 
\end{proof}
\begin{cor}
$GM(1)$ is the number of linearly independent ways there are to assign
current flows to a graph in such a way that they satisfy Kirchoff's
law of circuits.
\end{cor}

\begin{proof}
If we interpret $G$ as an electrical circuit and the coordinate $\mathbf{v}_{k\to l}$
as the current flow in the direction of $k$ toward $l$, then (\ref{eqn:system-1})
is exactly equivalent to Kirchoff's law.
\end{proof}
For our treatment of $\lambda=-1$, recall that $D+A$ is called the
\emph{signless Laplacian} of $G$. The proofs of the two following
propositions are similar to those of Propositions \ref{pro:am-plus1}
and \ref{pro:gm-plus1}.
\begin{prop}
Let $G$ have at least two cycles. Then $AM(-1)=m-n+1$ if $G$ is
bipartite and $AM(-1)=m-n$ if $G$ is not bipartite.
\end{prop}

\begin{proof}
Define $f(u)$ as in Proposition \ref{pro:am-plus1} and observe that
$f(-1)=\det\left(D+A\right)$, where $D+A$ is called the signless
Laplacian of $G$. It is known that $D+A$ is singular if and only
if $G$ is bipartite \cite{cvetkovic2007signless}. Therefore, if
$G$ is not bipartite, $f(-1)\neq0$ and $AM(-1)=m-n$ due to the
Ihara-Bass formula (\ref{eqn:ihara-bass}). If $G$ is bipartite,
$f(-1)=0$ and $AM(-1)\geq m-n+1$. In this case, write $f(u)=(1+u)g(u)$
and note $f'(-1)=g(-1)$. To finish, we show $g(-1)\neq0$. Let the
partition of the node set be $U_{1}$ and $U_{2}$ and define the
vector $\mathbf{v}$ by putting $\mathbf{v}_{i}=1$ if $i\in U_{1}$
and $\mathbf{v}_{j}=-1$ if $j\in U_{2}$. A similar procedure as
in Proposition \ref{pro:am-plus1} shows that 
\[
g(-1)=f'(-1)=\Tr\left(\adj\left(D+A\right)\left(2I-D\right)\right)=\mathbf{v}^{T}\left(2I-D\right)\mathbf{v}= \eta \left( 2n-2m \right) \neq 0,
\]
for some nonzero number $\eta$.
\end{proof}
\begin{prop}
Let $G$ have at least two cycles. Then $GM(-1)=m-n+1$ if $G$ is
bipartite and $GM(-1)=m-n$ if $G$ is not bipartite.
\end{prop}

\begin{proof}
Inspection of (\ref{eqn:eigenvector-in-two-directions}) when $\lambda=-1$
and an argument similar to that in Proposition \ref{pro:gm-plus1}
shows that if $B\mathbf{v}=-\mathbf{v}$ then 
\begin{equation}
\begin{cases}
\nbcent[l]=0 & \text{for each node }l,\\
\mathbf{v}_{k\to l}=\mathbf{v}_{l\to k} & \text{for each edge \ensuremath{k-l}}.
\end{cases}\label{eqn:system-minus-1}
\end{equation}
To fix ideas, suppose the cycle $\mathcal{C}=x\to y,y\to z,z\to t,t\to z$
exists in $G$, and consider the vector $\mathbf{v}$ with $\mathbf{v}_{x\to y}=\mathbf{v}_{y\to x}=\mathbf{v}_{z\to t}=\mathbf{v}_{t\to z}=1$
and $\mathbf{v}_{y\to z}=\mathbf{v}_{z\to y}=\mathbf{v}_{t\to x}=\mathbf{v}_{x\to t}=-1$,
so that $\mathbf{v}$ satisfies Equation (\ref{eqn:system-minus-1}).
(See Figure (\ref{fig:tree-circle})(c) for an example.) In general,
if $\mathcal{C}$ has even length, $\mathbf{v}$ will satisfy (\ref{eqn:system-minus-1}).
The dimension of the space spanned by the even-length cycles has been
studied in \cite{rowlinson2004,simic2015,godsil2013algebraic}, and
it is known to be $m-n$ when $G$ is not bipartite and $m-n+1$ when
it is.
\end{proof}
\begin{cor}
\label{cor:eigenspace-basis-real}The eigenspace of $\lambda=1$ admits
a basis where each element is supported on a different cycle (any
cycle in the graph, not only pendants or collars or bracelets), while
the eigenspace of $\lambda=-1$ admits a basis where each element
is supported on a different cycle of even length.\qed
\end{cor}

\begin{rem*}
We can use our knowledge of the multiplicities of unit eigenvalues
to study the poles of the so-called Ihara-Zeta function through the
Ihara-Bass formula (\ref{eqn:ihara-bass}) as well as other matrices
that may be associated to the underlying graph. For example, with
$f(u)$ as defined in Proposition \ref{pro:am-plus1}, evaluating
$f(i)$ yields that the matrix $A-iD$ has nullity equal to the number
of collars of size $4$ in the graph. In the future, it will be interesting
to see if this ``complex Laplacian'' matrix $A-iD$ holds any more
interesting information about the graph.
\end{rem*}

\subsubsection{Examples\label{sec:examples-unit}}

In Figure \ref{fig:tree-circle}, panel (b) shows eigenvectors $\mathbf{v},\mathbf{u}$
each of which is supported on a single cycle, and correspond to third
roots of unity. Panel (c) shows eigenvectors $\mathbf{w},\mathbf{z}$,
each of which is supported on a single cycle, corresponding to fourth
roots of unity. 

Consider a graph with a pendant of size $3$. Since the pendant has
six directed edges, its existence is associated to six NB-eigenvalues.
Three of them are the third roots of unity, as per the results of
this section. The other three may or may not be roots of unity, and
the corresponding eigenvectors will in general not be supported on
the pendant. This is illustrated in Figure \ref{fig:examples-unit}.
Panel (a) shows a graph with one cycle and non-empty $1$-shell; its
eigenvalues are described by Sections \ref{sec:1-shell} and \ref{sec:1-cycle}.
Panel (b) shows the same graph with one new edge added, forming a
new pendant of size~$3$. The multiplicities of the third roots of
unity equal the number of pendants of size $3$ in this graph. Further,
the other three eigenvalues associated to the addition of the new
pendant are in fact the fundamental sixth roots of unity, corresponding
to the formation of a bracelet of size $6$. The corresponding eigenvectors
are not supported on either pendant, but on the whole graph. The leading
eigenvalues of this graph are explained by Theorem \ref{thm:leading}.
Panel (c) shows the same graph as in (a) but with two new edges, forming
now a collar of size $4$. The only unit eigenvalues are those corresponding
to the pendant and the collar, all other eigenvalues are outer or
leading.

In summary, by adding a new collar or pendant of size $r$ to an arbitrary
graph, there will always be $r$ new eigenvalues that are $r^{th}$
roots of unity, as well as a new eigenvalue equal to $-1$. In some
cases, as in Figure \ref{fig:examples-unit}(b), the other $r-1$
new eigenvalues will be roots of unity as well, of some order not
necessarily $r$. In other other cases, as in Figure \ref{fig:examples-unit}(c),
those other eigenvalues are not unitary. Studying these $r-1$ eigenvalues,
as well as what happens to all the previous ones, is an interesting
direction of future research.

We now illustrate some further facts about non-leaky vectors. In the
graph of Figure \ref{fig:examples-unit}(c), let $\mathbf{v}_{1},\mathbf{v}_{2}$
be such that $B\mathbf{v}_{1}=j\mathbf{v}_{1}$ and $B\mathbf{v}_{2}=i\mathbf{v}_{2}$
and $\mathbf{v}_{1}$ is supported on the pendant and $\mathbf{v}_{2}$
is supported on the collar. Note that $\mathbf{v}\coloneqq\mathbf{v}_{1}+\mathbf{v}_{2}$
is non-leaky and it satisfies $B^{12}\mathbf{v}=\mathbf{v}$, but
it is not an eigenvector, in accordance with Lemma \ref{lem:non-leaky-combination}.
Thus, non-leaky vectors are not necessarily always eigenvectors, even
if their support can be decomposed in different cycles; cf. Corollary
\ref{cor:eigenspace-basis}.

Now consider the graph in Figure \ref{fig:karate}(b). Following Corollary
\ref{cor:eigenspace-basis}, there is a basis of the eigenspace corresponding
to $\lambda=i$ such that each element of a basis is supported on
a different collar of length $4$. In this case, though there are
three nodes of degree $2$, there are only two linearly independent
such collars, as any two of them overlap in exactly three edges. Thus,
the collars giving rise to the basis are all different, but they may
be overlapping. Note the graph in Figure \ref{fig:karate}(b) is a
subgraph of the graph shown in panel (a) of the same Figure, and thus
the basis corresponding to the eigenspace of $\lambda=i$ in this
graph also consists of overlapping collars. Finally, note that each
collar giving rise to these bases contains at least one unique node
of degree $2$; this is the node used in case $2$ of the proof of
Theorem \ref{thm:root}. 

\begin{figure}
\includegraphics{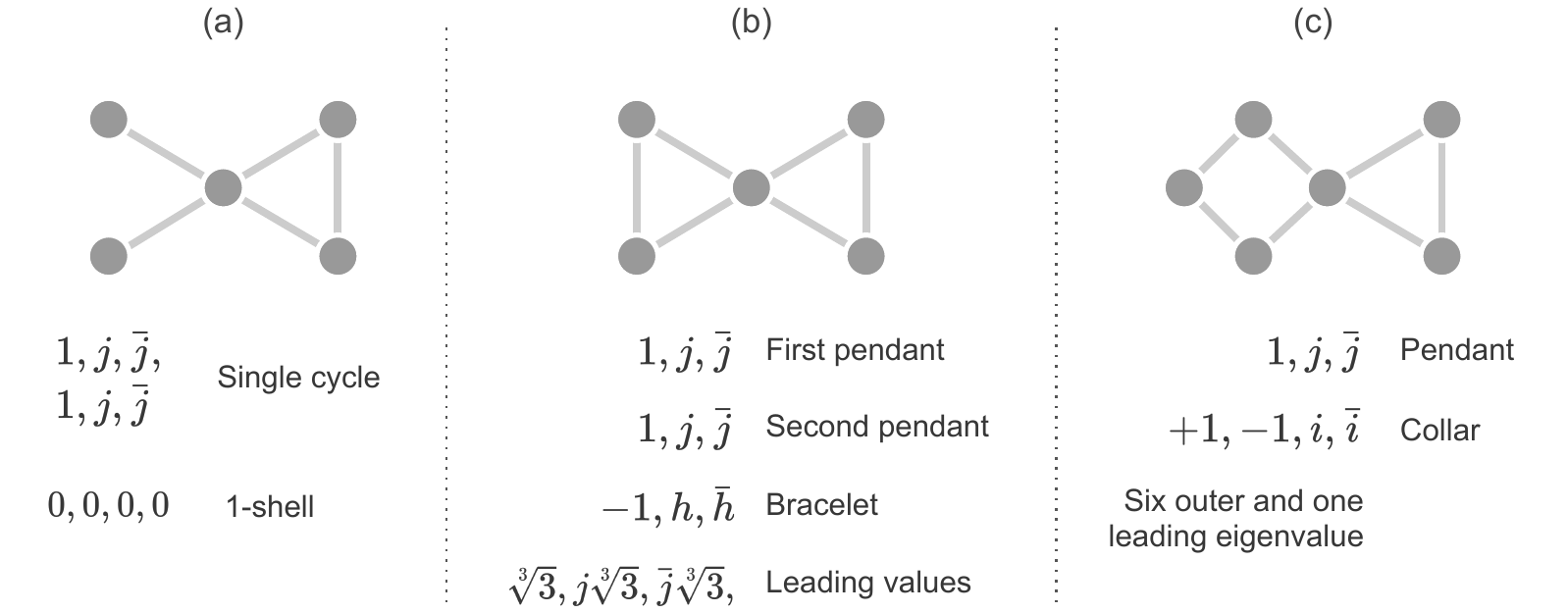}

\caption{\label{fig:examples-unit}Eigenvalues of example graphs. See Section
\ref{sec:examples-unit} for discussion. Here, $i^{2}=-1$, $j\protect\coloneqq\frac{-1+i\sqrt{3}}{2}$,
$h\protect\coloneqq\frac{1+i\sqrt{3}}{2}$. Adding a new collar or
pendant to the graph in (a) generates six new eigenvalues. In (b),
all six are unitary, owing to the fact that a bracelet of size $6$
has been formed. In (c), only three of the new eigenvalues are unitary.}
\end{figure}

\subsection{The outer eigenvalues: a conjecture\label{sec:outer}}

In our experience, the eigenvalues with $|\lambda|=1$ or $\lambda=0$
are the only eigenvalues we have found in practice to have multiplicity
greater than $1$ in random graphs and real networks. (The case of
leading eigenvalues $|\lambda|=\rho$ is treated in the next section).
In the case of outer eigenvalues, we believe results similar to the
case of random matrices \cite{erdos2012,tao2017} will hold for the
NB-matrix. In particular, it is know that some ensembles of random
matrices have simple spectrum. If that is the case for the outer NB-eigenvalues,
the only eigenvalue that can be expected to be defective is $\lambda=0$
(as per Proposition \ref{pro:zero}). In view of this observation,
we present the following conjecture on the diagonalizability of the
NB-matrix.
\begin{conjecture*}
The NB-matrix of $G$ is diagonalizable if and only if $G$ has empty
$1$-shell. 
\end{conjecture*}

\subsection{The leading eigenvalues\label{sec:leading}}

Let $G$ be a graph whose NB-matrix has spectral radius $\rho$. Kotani
and Sunada show that $\rho=1$ if and only if $G$ is a cycle graph.
Further, they fully characterize those eigenvalues with $|\lambda|=\rho>1$
for graphs with more than one cycle in Theorem 1.4 of \cite{Kotani2000}
using the language of graph Zeta functions. For completeness, here
we paraphrase their theorem in the language of the NB-matrix.
\begin{thm}
[from \cite{Kotani2000}]\label{thm:leading}Let $G$ be a md$2$
graph with at least $2$ cycles. Let $\nu$ be the greatest common
divisor of the set of lengths of all NB-cycles. Then, every $\lambda=\rho\exp\left(2\pi ik/\nu\right)$
for $k=1,\ldots,\nu-1$ is a NB-eigenvalue of $G$ with multiplicity
$1$. \qed
\end{thm}

\begin{rem*}
The proof of this Theorem is a consequence of Lemma 2.1 of \cite{Kotani2000}.
Essentially, it is a consequence of applying the Perron-Frobenius
theorem to the NB-matrix. Indeed, in the language of Perron-Frobenius
theory, the NB-matrix of $G$ is always irreducible with period $\nu$.
\end{rem*}
\begin{cor}
[from \cite{Kotani2000}] Suppose $G$ has minimum degree at least
$3$. 
\begin{enumerate}
\item If $G$ is bipartite, then $\nu=2$ and there are only two leading
eigenvalues, namely $\rho$ and $-\rho$.
\item If $G$ is not bipartite, then $\nu=1$ and there is only one leading
eigenvalue, namely $\rho.$ In this case, we call $\lambda=\rho$
\emph{the Perron eigenvalue}. \qed
\end{enumerate}
\end{cor}

\begin{rem*}
The Corollary is a consequence of Theorem 1.5 of \cite{Kotani2000}.
Recall from Section \ref{sec:zero-one} that the $1$-shell does not
affect the non-zero eigenvalues. Therefore, the corollary can be slightly
strengthened by changing the assumption that $G$ has minimum degree
at least $3$ to the assumption that the $2$-core of $G$ has minimum
degree at least $3$.
\end{rem*}

\section{\label{sec:diagonalizability}Diagonalizability}

Let $G$ be a a graph with NB-matrix $B$ and empty $1$-shell. If
$B$ can be written in diagonal form as $B=Z\Lambda Z^{-1}$ or, equivalently,
as $B^{*}=\left(Z^{-1}\right)^{*}\Lambda^{*}Z^{*}$, where $\Lambda$
is a diagonal matrix, then the columns of $Z$ contain the right eigenvectors
of $B$ while the rows of $Z^{-1}$ contain the left eigenvectors.
Since $B$ is not normal, we know that $Z$, if it exists, cannot
be unitary, but we may still find relationships among the columns
of $Z$ and the rows of $Z^{-1}$, i.e. between the right and left
eigenvectors.

To do so, we use a special kind of symmetry exhibited by $B$, sometimes
called  \emph{PT-symmetry} \cite{bordenave2015}. Let $P$ be the
operator defined as $P\chi^{i\to j}=\chi^{j\to i}$. It is readily
seen that this operator is involutory ($P^{2}=I$), symmetric ($P^{*}=P$),
and orthogonal ($P^{-1}=P$). We can use Equation (\ref{eqn:nbm})
to prove that $PB$ is symmetric \cite{bordenave2015}. The right
and left NB-eigenvectors are related through $P$.
\begin{lem}
Let $\mathbf{v}$ be a right eigenvector of $B$ with eigenvalue $\lambda$.
Then the row vector $\mathbf{v}^{T}P$ is a left eigenvector of $B$
of eigenvalue $\lambda$.
\end{lem}

\begin{proof}
That $\mathbf{v}$ is a right eigenvector implies that $\mathbf{v}^{T}B^{T}=\lambda\mathbf{v}^{T}$.
Since $PB$ is symmetric and $P^{2}=I$, we have $B^{T}=PBP$. Use
these two equations and multiply by $P$ again to find $\mathbf{v}^{T}PB=\lambda\mathbf{v}^{T}P$.
\end{proof}
\begin{rem*}
Importantly, in the proof of this Lemma we take the transpose $\left(B^{T}\right)$
and not the adjoint $\left(B^{*}\right)$. This is immaterial for
$B$ since it is a real matrix and thus $B^{*}=B^{T}$; but it is
important for both $\mathbf{v}$ and $\lambda$. Taking the adjoint
leads to the fact that $\mathbf{\bar{v}}^{T}$ is a left eigenvector
of $\bar{\lambda}$, which is a fact that holds for any real matrix,
without the assumption of PT-symmetry.
\end{rem*}
As both $Z^{-1}$ and $Z^{T}P$ contain left eigenvectors in the rows,
we are tempted to ask whether $Z^{-1}=Z^{T}P$. If this were the case,
it would imply $P=ZZ^{T}$, which in turn implies that $P$ is positive
semi-definite. However, this is false as $P$ has eigenvalues $\pm1$.
What then can be said about $Z^{-1}$? We answer this question in
two parts.

First, let $R$ be a matrix whose columns are a maximal set of right
eigenvectors corresponding to unit eigenvalues. Suppose $B_{R}$ is
the restriction of $B$ to the space spanned by the columns of $R$.
We have that $B_{R}$ is a unitary matrix (since all its eigenvalues
are unitary) and therefore it is unitarily diagonalizable. In fact,
we have $B_{R}=RUR^{*}$, where $U$ is a diagonal matrix with the
unit eigenvalues, as well as $RR^{*}=I$. 

Second, consider the eigenvectors of non-unit eigenvalues. Suppose
$B\mathbf{v}=\lambda_{1}\mathbf{v}$ and $B\mathbf{u}=\lambda_{2}\mathbf{u}$.
Since every left eigenvector is orthogonal to a right eigenvector
of a different eigenvalue, we have $\mathbf{v}^{T}P\mathbf{u}=0$
whenever $\lambda_{1}\neq\lambda_{2}$. We refer to this property
as \emph{$P$-orthogonality}. In particular, if $\lambda_{1}$ is
a simple eigenvalue then $\mathbf{v}$ will be $P$-orthogonal to
every other eigenvector. Now let $Q$ be a matrix whose columns are
a maximal set of right eigenvectors corresponding to the non-unit
eigenvalues, and let $B_{Q}$ be the restriction of $B$ to the space
spanned by the columns of $Q$. If our conjecture on the simplicity
of outer eigenvalues holds, we will have $B_{Q}=QVQ^{T}P$, where
$V$ is a diagonal matrix containing all the non-unit eigenvalues.
In addition, the columns of $Q$ can be chosen such that $Q^{T}PQ=I$,
and in this case we say $Q$ is $P$-orthogonal.

In all, when the conjecture on the simplicity of outer eigenvalues
is true, we can write\begin{equation}\label{eqn:diagonalization}
B = Z \Lambda Z^{-1} = 
\left(
\begin{array}{c;{2pt/2pt}c}
  Q & R
\end{array}
\right)
\times
\left(
\begin{array}{c;{2pt/2pt}c}
  V & \\
  \hdashline[2pt/2pt] & U \\
\end{array}
\right)
\times
\left(
\begin{array}{c}
  Q^T P \\
  \hdashline[2pt/2pt] R^*
\end{array}
\right)
,
\end{equation}where $Q^{T}PQ=I$, $R^{*}R=I$, $Q^{T}PR=0$, and $R^{*}Q=0$.

\section{\label{sec:application}Application: the Perron eigenvalue after
node addition}

In \cite{torres2020}, the authors investigated the following question
relating to the Perron eigenvalue of the NB-matrix of a graph undergoing
node addition.\footnote{Reference \cite{torres2020} is stated in terms of node removal, but
all the arguments therein apply to the present setting of node addition
as well.} Let $G$ be a graph and add a new node $c$ to form a new graph $G^{c}$;
see Figure \ref{fig:add-node}. Let $B,B^{c}$ be the corresponding
NB-matrices, and $\lambda,\lambda_{c}$ be the corresponding Perron
eigenvalues. In \cite{torres2020}, the authors use heuristics to
bound the difference $\lambda_{c}-\lambda$, which they call \emph{the
eigen-drop, }and develop algorithms exploiting these heuristics to
find the node that generates the largest difference. Importantly,
their arguments depend on the diagonalizability of $B$ and $B^{c}$,
which has been established here in previous sections. Our present
goal is to rigorously show that $\lambda_{c}>\lambda$, a fact that
was only assumed in \cite{torres2020}.

\begin{figure}
\begin{centering}
\includegraphics[scale=0.9]{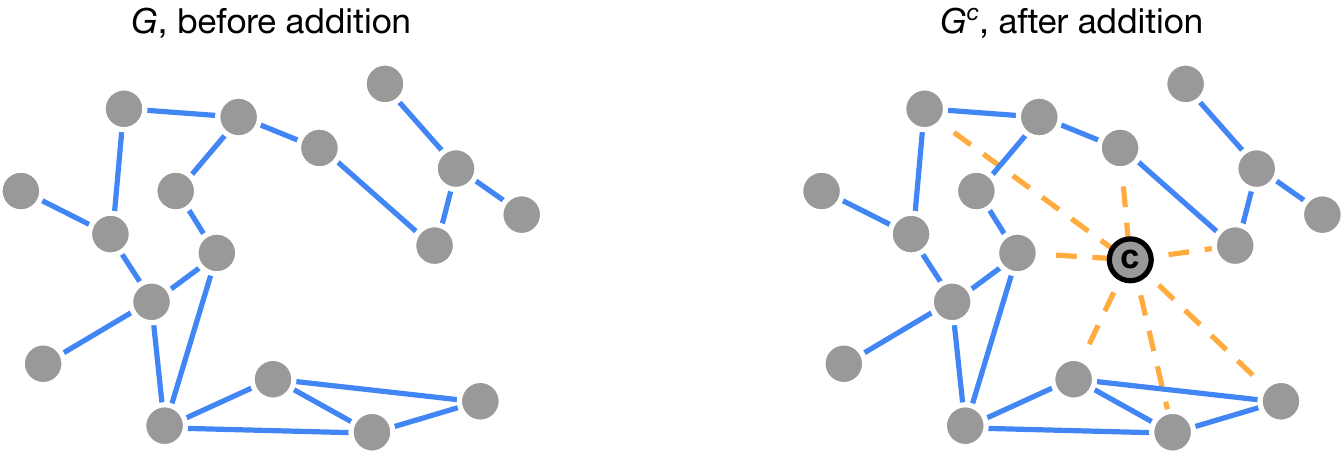}
\par\end{centering}
\begin{flalign*}
\quad\quad\quad\quad\quad
\left( \begin{array}{ccc}  &   & \\  & B & \\  &   & \\   \end{array} \right)
\quad\quad\quad\quad\quad\quad\quad\quad\quad\quad\quad
B^c = \left( \begin{array}{c;{2pt/2pt}c} \begin{array}{ccc}  &   & \\
& B & \\  &   & \\   \end{array}  & D \\  \hdashline[2pt/2pt] E & F \\ \end{array} \right) \hspace{-2.5em} \begin{tabular}{l} $\left.\phantom{\begin{array}{c} \\ D \\ \\ \end{array}}\right\} 2m $ \\ $\left.\phantom{\begin{array}{c} F \end{array}}\right\} 2d $ \end{tabular}
\end{flalign*}

\caption{\textbf{\label{fig:add-node}Top:} Construction of $G^{c}$ from $G$
by adding a new node $c$ with degree $d$. All edges incident to
$c$ are in dashed yellow lines. \textbf{Bottom:} corresponding NB-matrices.
Adapted with permission from \cite{torres2020}.}
\end{figure}

We quickly recall some of the necessary results from \cite{torres2020}.
Due to space limitations we do not reproduce the proofs here. In what
follows, let $G$ have $n$ nodes and $m$ edges. Construct $G^{c}$
by adding a new node $c$ of degree $d$ to $G$. Accordingly, $B$
is a square matrix of size $2m$ and $B^{c}$ is a square matrix of
size $2m+2d$. We can write $B^{c}$ in block form as shown in the
bottom right of Figure \ref{fig:add-node}, where $B$ is the NB-matrix
of the original graph, and $F$ is indexed in the rows and columns
by yellow edges. Accordingly, $D$ is indexed in the rows by blue
edges and in the columns by yellow edges, and vice versa for $E$.
Note that all of $B,D,E,F$ are sub-matrices of $B^{c}$ and thus
we know their general element is given by Equation (\ref{eqn:nbm}).
In \cite{torres2020} it was established that $F^{2}=0$ and $DE=0$.
Now define $X\coloneqq DFE$ and note $X_{k\to l,i\to j}=a_{ck}a_{cj}(1-\delta_{jk})$.
Following the top right of Figure \ref{fig:add-node}, $X$ is a binary
matrix that keeps track of NB-walks that consist of four edges of
colors blue-yellow-yellow-blue. Note these are precisely those paths
formed by the addition of the new node $c$ and thus $X$ will be
essential to our discussion. Finally, in \cite{torres2020} it was
also shown that 
\[
\det\left(B^{c}-tI\right)=t^{2d}\det\left(B-tI+\frac{X}{t^{2}}\right),
\]
whenever $t$ is not an eigenvalue of $F$, i.e. whenever $t\neq0$
since $F$ is nilpotent.

Now define $Y(t)\coloneqq\left(B-tI\right)^{-1}$ and factor it to
get
\begin{equation}
\det\left(B^{c}-tI\right)=t^{2d}\det\left(B-tI\right)\det\left(I+\frac{Y(t)X}{t^{2}}\right).\label{eqn:yx}
\end{equation}
Here, $Y(t)$ is called the \emph{resolvent }of $B$ and it is of
utmost importance to the theory of random matrices, where its trace
is called the \emph{Stjelties transform }of $B$. Observe from Equation
(\ref{eqn:yx}) that every non-zero eigenvalue $t$ of $B^{c}$ that
is not an eigenvalue of $B$ must satisfy that $\det\left(I+\frac{Y(t)X}{t^{2}}\right)=0$
and therefore $-t^{2}$ must be an eigenvalue of $Y(t)X$. In the
following lines, we give $Y(t)$ a suitable form, which will then
allow us to show that there exists a real eigenvalue of $B^{c}$,
namely its Perron eigenvalue $\lambda_{c}$, such that $\lambda_{c}>\lambda$.
We proceed in several steps:
\begin{enumerate}
\item Use the assumption of diagonalizability of $B$ to rewrite its resolvent
$Y(t)$.
\item Apply the Perron-Frobenius theorem to $Y(t)X$ to find its Perron
eigenvalue $y(t)$.
\item Define the auxiliary matrix $H\coloneqq H(t)$ which also has $y(t)$
as an eigenvalue.
\item Apply Gershgorin's Disk theorem to $H$ to show that at some $t_{0}$
it holds that $y(t_{0})=-t_{0}^{2}$, as desired. This $t_{0}$ will
in fact be $\lambda_{c}$, the Perron eigenvalue of $B^{c}$.
\end{enumerate}

\paragraph{Step 1: Rewriting the resolvent.}

Suppose $B$ is diagonalizable with $R$ a matrix of right eigenvectors
as columns and $L$ a matrix with left eigenvectors as rows.\footnote{Do not confuse this $R$ matrix with that used in Section \ref{sec:diagonalizability}.
In fact, if our conjecture on outer eigenvalues is true, the $R$
matrix used in this Section takes the form of the $Z$ matrix in Section
\ref{sec:diagonalizability}, and $L$ becomes $Z^{-1}$.} Define $T$ as the diagonal matrix with $T_{ii}=T_{ii}(t)\coloneqq1/\sqrt{\lambda_{i}-t}$
for $i=1,\ldots,2m$ where $\lambda_{i}$ are the eigenvalues of $B$
sorted according to decreasing modulus; we continue referring to $\lambda_{1}$
as simply $\lambda$. If two eigenvalues have the same modulus, sort
them arbitrarily. Then we can write 
\begin{equation}
Y(t)=\left(B-tI\right)^{-1}=\left(L\Lambda R-tI\right)^{-1}=R\left(\Lambda-tI\right)^{-1}L=RT^{2}L=\sum_{i}\frac{\mathbf{v}_{i}^{R}\mathbf{v}_{i}^{L}}{\lambda_{i}-t},
\end{equation}
where $\Lambda$ contains all eigenvalues in order, and $\mathbf{v}_{i}^{R},\mathbf{v}_{i}^{L}$
are right and left eigenvectors corresponding to $\lambda_{i}$, respectively,
chosen such that $\mathbf{v}_{i}^{L}\mathbf{v}_{i}^{R}=1$\@. As
mentioned above, we are looking for an eigenvalue of $Y(t)X=RT^{2}LX$
that equals $-t^{2}$. In what follows we drop the dependence on $t$
when possible for ease of notation.
\begin{lem}
[\bf{Step 2: Apply the Perron-Frobenius theorem}]\label{lem:yx-negative}Fix
$t$ with $|t|>\lambda$ and let $\rho(t)$ be the spectral radius
of $Y(t)X$. Then $Y(t)X$ has a simple real negative eigenvalue $y(t)$
such that $y(t)=-\rho(t)$.
\end{lem}

\begin{proof}
Since $|t|>\lambda_{1}$, we can use the Neumann series to inspect
each entry of $YX$: for any two oriented edges $e_{1},e_{2}$ we
have
\begin{equation}
\left(YX\right){}_{e_{1}e_{2}}=-\sum_{k=0}^{\infty}\frac{1}{t^{k+1}}\left(B^{k}X\right)_{e_{1}e_{2}}.
\end{equation}

Since the graph is connected, for each entry $e_{1}e_{2}$ there exists
a $k$ such that $\left(B^{k}\right)_{e1e2}$ is positive. Therefore,
$\left(YX\right)_{e_{1}e_{2}}$ is negative unless every element in
the $e_{2}$ column of $X$ is zero. Thus, $YX$ is non-positive.
Furthermore, after reordering its columns, $YX$ has the block form
\begin{equation}
YX=\left(\begin{array}{cc}
Y_{1} & 0\\
Y_{2} & 0
\end{array}\right),
\end{equation}
for some square matrix $Y_{1}$ and rectangular matrix $Y_{2}$. This
implies that the eigenvalues of $YX$ are equal to the eigenvalues
of $Y_{1}$. But the entries of $Y_{1}$ are all strictly negative,
thus the Perron-Frobenius theorem implies that there is a negative
real number $y=y(t)$ such that it is a simple eigenvalue of $YX$
equal to $-\rho$.
\end{proof}

\paragraph*{Step 3: Define the auxiliary matrix.}

For the purpose of bounding $y$, we consider the matrix $H=H(t)=TLXRT.$
Note that $H$ and $YX$ are cyclic permutations of the same
matrix product and therefore they have the same eigenvalues. In particular
$y$ is an eigenvalue of $H$, for each $t$. 
\begin{thm}
[\bf{Step 4: Apply Gershgorin's Disk theorem}] \label{thm:increase}There
exists a real number $\lambda_{c}$ with $\lambda_{c}>\lambda$ such
that $-\lambda_{c}^{2}$ is an eigenvalue of $Y(\lambda_{c})X$ and
$\lambda_{c}$ is an eigenvalue of $B^{c}$.
\end{thm}

\begin{proof}
Put $r_{i}\coloneqq\sum_{j\neq i}\left|H_{ij}\right|$ and define
the $i^{th}$ Gershgorin disk as $D_{i}\coloneqq\left\{ z:\left|z-H_{ii}\right|\leq r_{i}\right\} $.
(Note here that both the center and the radius of each disk $D_{i}$
are changing as a function of $t$.) Gershgorin's disk theorem says
that all eigenvalues of $H$ must be contained in the union of all
$D_{i}$. Furthermore, a strengthened version of the theorem says
that if one of the disks is isolated from the rest, then it must contain
exactly one eigenvalue. To prove the existence of $\lambda_{c}$,
we proceed in three steps, as illustrated in Figure \ref{fig:disks}:
\begin{enumerate}
\item First we show that for some small $\epsilon>0$, at $t=\lambda+\epsilon$,
$D_{1}$ is disjoint from all other circles and that it must contain
$y$, the least eigenvalue of $H$. 
\item Second, also at $t=\lambda+\epsilon$, we prove that every real number
in $D_{1}$ is less than $-t^{2}$.
\item Finally, we show that as $t$ goes to $\infty$, every number inside
each $D_{i}$ must be smaller in magnitude than $-t^{2}$. 
\end{enumerate}
Since $y$ is a real continuous function of $t$, these three assertions
imply that at some point $\lambda_{c}$ in $[\lambda+\epsilon,+\infty)$,
we must have $y(\lambda_{c})=-\lambda_{c}^{2}$, and therefore the
theorem follows. We address all three claims in turn with the following
inequalities. Write $\alpha_{ij}\coloneqq\mathbf{v}_{i}^{L}X\mathbf{v}_{j}^{R}$
such that we can write 
\begin{equation}
D_{i}=\left\{ z:\left|z-H_{ii}\right|<r_{i}\right\} =\left\{ z:\left|z-\frac{\alpha_{ii}}{\lambda_{i}-t}\right|\le\frac{1}{\sqrt{t-\lambda_{i}}}\sum_{j\neq i}\left|\frac{\alpha_{ij}}{\sqrt{t-\lambda_{j}}}\right|\right\} .
\end{equation}

For step (1), consider $D_{1}$ when $t$ approaches $\lambda$ from
the right. Write $H_{11}+\delta r_{1}$ for an arbitrary number inside
$D_{1}$ where $\delta$ is a complex number with $\left|\delta\right|\le1$.
Similarly, write $H_{ii}+\delta'r_{i}$ for an arbitrary element in
$D_{i}$, $i\neq1$, where $\delta$ is a complex number with $\left|\delta'\right|\le1$.
Then we have 
\begin{equation}
\left|H_{11}+\delta r_{1}\right|=\left|\frac{\alpha_{11}}{\lambda_{1}-t}+\frac{\delta}{\sqrt{t-\lambda_{1}}}\sum_{j\neq1}\left|\frac{\alpha_{1j}}{\sqrt{t-\lambda_{j}}}\right|\right|\ge\left|\frac{\alpha_{ii}}{\lambda_{i}-t}+\frac{\delta'}{\sqrt{t-\lambda_{i}}}\sum_{j\neq i}\left|\frac{\alpha_{ij}}{\sqrt{t-\lambda_{j}}}\right|\right|=\left|H_{ii}+\delta'r_{i}\right|.
\end{equation}
The inequality holds regardless of $\delta,\delta'$ when $\lambda_{1}\leftarrow t$
and therefore $D_{1}$ is disjoint from all other disks. Since $y$
is the least eigenvalue of $H$, $D_{1}$ contains $y$ and no other
eigenvalue of $H$.

For (2), consider an arbitrary real number inside $D_{1},$ namely
$H_{11}+\delta r_{1}$ for some real $\delta\in[-1,1]$. We have

\begin{equation}
H_{11}+\delta R_{1}=\frac{\alpha_{11}}{\lambda_{1}-t}+\frac{\delta}{\sqrt{t-\lambda_{1}}}\sum_{j\neq1}\left|\frac{\alpha_{1j}}{\sqrt{t-\lambda_{j}}}\right|\leq-t^{2},
\end{equation}
which holds when $t$ is sufficiently close to, but larger than, $\lambda_{1}$
and whenever $\alpha_{11}$ is non-negative. But applying the Perron-Frobenius
theorem on $B$ implies that $\mathbf{v}_{1}^{R}$ and $\mathbf{v}_{1}^{L}$
are both strictly positive and therefore $\alpha_{11}=\mathbf{v}_{1}^{L}X\mathbf{v}_{1}^{R}$
is non-negative.\footnote{This $\alpha_{11}$ is what the authors of \cite{torres2020} called
the \emph{X-non-backtracking centrality} of the newly added node $c$.}

For (3), when $t\to\infty$ we have 
\begin{equation}
-t^{2}\leq\frac{\alpha_{ii}}{t-\lambda_{i}}+\frac{\delta}{\sqrt{t-\lambda_{i}}}\sum_{j\neq i}\left|\frac{\alpha_{ij}}{\sqrt{t-\lambda_{j}}}\right|,
\end{equation}
for each real $\delta\in[-1,1]$. This finishes the proof.
\end{proof}
This Theorem establishes a weak version of eigenvalue interlacing
for the NB-matrix. Indeed, after adding (or removing) the rows and
columns incident to the same node $c$, the Perron eigenvalue behaves
as expected: it can only increase when a new node is added to the
graph, and it can only decrease when a node is removed from the graph.
However, the other eigenvalues do not seem to behave similarly. It
remains an open question if more general versions of interlacing apply
to the NB-matrix.

\begin{figure}
\begin{centering}
\includegraphics{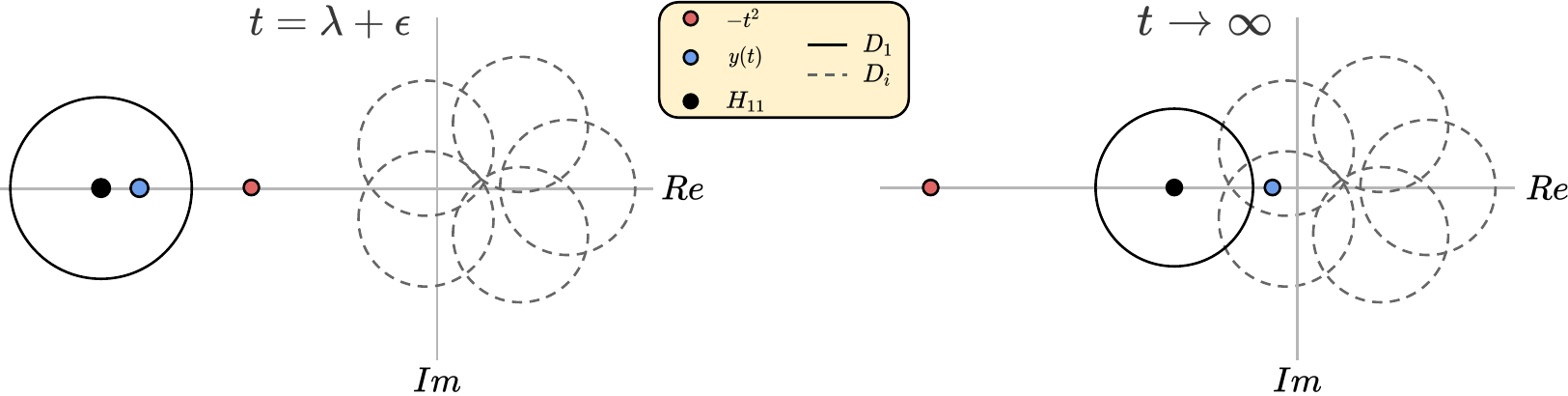}
\par\end{centering}
\caption{\emph{\label{fig:disks}}\textbf{Left:}\emph{ }when $t\!=\!\lambda\!+\!\epsilon$,
$y(t)$ lies inside $D_{1}$ which in turn lies to the left of $-t^{2}$
and is disjoint from the rest. \textbf{Right:} when $t\to\infty$,
$y(t)$ lies in some of the $D_{i}$, all of which lie to the right
of $-t^{2}$.}
\end{figure}

\section*{Acknowledgements}

This work was supported by NSF IIS-1741197. L.T.~thanks Gabor Lippner for many invaluable conversations, and Blevmore Labs for consulting services regarding the construction
Figures \ref{fig:nbm-doodle} and \ref{fig:bbt}.

\bibliographystyle{plain}
\bibliography{references}

\appendix
\counterwithin{thm}{section}

\section{Eigenvalues and eigenvectors\label{app:eigen}}

Let $B$ be an arbitrary square matrix, and $\lambda,\mathbf{v}$
be such that $B\mathbf{v}=\lambda\mathbf{v}$. Here, $\lambda$ is
called an \emph{eigenvalue }of $B$ and $\mathbf{v}$ a \emph{right
eigenvector }of $B$. If, on the other hand, we have $\mathbf{v}^{T}B=\lambda\mathbf{v}^{T}$
then $\mathbf{v}^{T}$ is called a \emph{left eigenvector} of $B$.
The \emph{characteristic polynomial} of $B$ is $\det\left(B-tI\right)$
and its roots are the eigenvalues of $B$. The \emph{algebraic multiplicity}
of $\lambda$, denoted $AM(\lambda)$, is the multiplicity of $\lambda$
as a root of the characteristic polynomial of $B$. The \emph{right
(or left) eigenspace} of an eigenvalue $\lambda$ is the linear subspace
spanned by all right (left) eigenvectors corresponding to $\lambda$.
The \emph{geometric multiplicity} of $\lambda$, $GM(\lambda)$, is
the dimension of the corresponding right eigenspace. It holds that
$AM(\lambda)\geq GM(\lambda)$, and when the inequality is strict,
$\lambda$ is called \emph{defective}. $B$ is \emph{diagonalizable
}when it can be written as $B=U\Lambda U^{-1}$, with $\Lambda$ a
diagonal matrix. Equivalently, $B$ is diagonalizable if the two multiplicities
of every eigenvalue coincide. $B$ is \emph{invertible} when $0$
is not an eigenvalue. $B$ is \emph{normal} if $BB^{*}=B^{*}B$, where
$B^{*}=\bar{B}^{T}$ is the conjugate transpose. The celebrated \emph{spectral
theorem} says that a matrix is diagonalizable by a unitary transformation,
that is $B=UDU^{*}$ with $UU^{*}=I$, if and only if $B$ is normal.

\section{$k$-cores\label{app:kcore}}

For an arbitrary graph $G$ and an integer $k$, the $k$-core of
$G$ is the maximal induced subgraph of $G$ where each node has degree
at least $k$. The $k$-core of $G$ can be obtained using the following
algorithm. First identify all the nodes whose degree is less than
$k$, and remove them from $G$. After this removal, the degree of
some other nodes may have dropped below $k$. Keep removing nodes
of degree less than $k$ until there are none. The resulting graph
is the $k$-core \cite{batagelj2003m,batageljZ11}. 

The NB-eigenvalues are tightly related to the $2$-core of $G$. The
$1$-shell of $G$ is made up of the nodes and edges not in the $2$-core,
i.e. all those nodes removed at some step in the aforementioned algorithm.
The $1$-shell of $G$ is always a forest, and, as such, it allows
the following decomposition. All nodes of degree $1$ in $G$, i.e.
the \emph{leaves}, as well as all edges incident to them, form the
$1^{st}$ layer of the $1$-shell. Identify all those nodes whose
degree drops to $1$ after removing the nodes and edges in the $1^{st}$
layer. These nodes, and the remaining edges incident to them, form
the $2^{nd}$ layer of the $1$-shell. The remaining layers of the
$1$-shell are defined inductively. In this way, every node and edge
in the $1$-shell belongs to exactly one of its layers. These definitions
imply, for example, that a tree has empty $2$-core, a tree is equal
to the subgraph induced by its $1$-shell, that $2$-cores have no
nodes of degree one, and that graphs in which all nodes have degree
at least $2$ have empty $1$-shell. In the main body we use these
equivalent properties without proof.

Now consider an undirected edge $u-v$ belonging to the $r^{th}$
layer of the $1$-shell. It is always the case that one of its endpoints
belongs to the $r^{th}$ layer and the other belongs to the $\left(r+1\right)^{th}$
layer. Further, if $u$ belongs to the $r^{th}$ layer, we say that
the oriented edge $u\to v$ is \emph{pointing inward}, while $v\to u$
is \emph{pointing outward}. Intuitively, outward edges are pointing
in the direction of the leaves of $G$, while inward edges point in
the direction of the $2$-core. Lastly, note that each oriented edge
is part of at least one NB-cycle if and only if it is inside the $2$-core.

\section{Technical lemmas}
\begin{lem}
\label{lem:nb-in-cent}Assume $B\mathbf{v}=\lambda\mathbf{v}$. Then
for any node $k$ we have 
\begin{equation}
\left(d_{k}-1\right)\nbcent=\lambda\,\incent.
\end{equation}
\end{lem}

\begin{proof}
Replacing (\ref{eqn:nb-centrality}) in (\ref{eqn:eigenvector-in-two-directions})
and summing over all neighbors of $k$ we obtain the result.
\end{proof}
\begin{lem}
\label{lem:non-leaky-unitary}Let $B\text{\ensuremath{\mathbf{v}}=\ensuremath{\lambda\mathbf{v}}}$
with $\lambda\neq0$. $\mathbf{v}$ is non-leaky if and only if $B^{*}B\mathbf{v}=BB^{*}\mathbf{v}=\mathbf{v}$
if and only if $\lambda$ is unitary.
\end{lem}

\begin{proof}
Since $\mathbf{v}$ is non-leaky, then $\incent=0$ for each $k$.
Since $\mathbf{v}$ is an eigenvector, Lemma \ref{lem:nb-in-cent}
implies that $\nbcent=0$ as well. By Lemma \ref{lem:bbt}, we have
$B^{*}B\mathbf{v}=BB^{*}\mathbf{v}=\mathbf{v}$. The converse is true
by definition. 

Multiply $B\text{\ensuremath{\mathbf{v}}=\ensuremath{\lambda\mathbf{v}}}$
by its conjugate transpose to get $\mathbf{v}^{*}B^{*}B\mathbf{v}=\bar{\lambda}\lambda\mathbf{v}^{*}\mathbf{v}$.
Assuming $B^{*}B\mathbf{v}=BB^{*}\mathbf{v}=\mathbf{v}$, we get $\bar{\lambda}\lambda=1$,
i.e. $\lambda$ is unitary. Now assume $\lambda$ is unitary. Then
the restriction of $B$ to the span of $\mathbf{v}$ is a unitary
operator and therefore $B^{*}B\mathbf{v}=BB^{*}\mathbf{v}=\mathbf{v}$
must hold true.
\end{proof}
\begin{lem}
\label{lem:cycle-degree-2}Let $G$ have minimum degree at least $2$
and suppose $\lambda,\mathbf{v}$ are such that $B\mathbf{v}=\lambda\mathbf{v}$.
Then there must exist a cycle $i_{1}\to i_{2}\to i_{3}\to\ldots\to i_{1}$
such that each $\mathbf{v}_{i_{r}\to i_{r+1}}$ is nonzero.
\end{lem}

\begin{proof}
Since $\mathbf{v}$ is nonzero, there must exist a nonzero component
$\mathbf{v}_{i\to j}$. But $\lambda\mathbf{v}_{i\to j}=\left(B\mathbf{v}\right)_{i\to j}=\sum_{k\neq j}a_{ij}\mathbf{v}_{k\to i}$,
which means there exists a $k$ such that $\mathbf{v}_{k\to i}\neq0$.
Apply the same argument to $\mathbf{v}_{k\to i}$ to obtain, say,
$\mathbf{v}_{l\to k}\neq0$. The walk constructed by iterating this
argument will never contain backtracks and can always continue to
be extended. We can keep adding edges to this walk until we pick an
edge that is already part of the walk. At this point, the walk must
contain a cycle in each of whose edges $\mathbf{v}$ is nonzero.
\end{proof}
\begin{lem}
\label{lem:non-leaky-combination}Let $\mathbf{v}$ be a non-leaky
vector. Then it must be the linear combination of eigenvectors each
of which corresponds to a unitary eigenvalue.
\end{lem}

\begin{proof}
Let $B_{L}$ be the restriction of $B$ to the space spanned by all
non-leaky vectors. Per Lemma \ref{lem:bbt}, we have that $B_{L}B_{L}^{*}=B_{L}^{*}B_{L}=I$,
that is $B_{L}$ is unitary. Therefore, there exists a basis of this
space comprised of eigenvectors of $B$ corresponding to unit eigenvalues.
\end{proof}
\begin{lem}
\label{lem:trace-cyclic}Given an arbitrary $n\times n$ matrix $X$
and a vector $\mathbf{v}\in\mathbb{R}^{n}$, we have $\Tr\left(\mathbf{v}\mathbf{v}^{T}X\right)=\mathbf{v}^{T}X\mathbf{v}$.
\end{lem}

\begin{proof}
Since the matrix $\mathbf{v}\mathbf{v}^{T}$ has rank one by definition,
then $\mathbf{v}\mathbf{v}^{T}X$ has rank at most one. Its rank is
zero if and only if $X\mathbf{v}=0$, and in this case we have $\Tr\left(\mathbf{v}\mathbf{v}^{T}X\right)=0=\mathbf{v}^{T}X\mathbf{v}$.
Now assume the rank of $\mathbf{v}\mathbf{v}^{T}X$ is one and define
$R=\frac{\mathbf{v}\mathbf{v}^{T}X}{\mathbf{v}^{T}X\mathbf{v}}$.
Note that $R$ is idempotent and therefore its rank equals its trace
(see e.g. \cite{petersen2012matrix}, Equation (423)). Thus
\[
1=\Tr\frac{\mathbf{v}\mathbf{v}^{T}X}{\mathbf{v}^{T}X\mathbf{v}},
\]
which finishes the proof.
\end{proof}

\end{document}